\documentclass{imsart}
\RequirePackage{natbib}
\usepackage{bbm}
\usepackage{amsmath,amssymb,amsthm}
\usepackage{graphics,epsfig}
\usepackage{hyperref}
\usepackage{natbib}
\usepackage{color}
\usepackage{graphicx}
\usepackage{caption}
\usepackage{subcaption}
\usepackage{float}
\usepackage{dsfont}
\usepackage{mathrsfs}
\usepackage{multirow}
\usepackage{comment}
\usepackage{bm}
\usepackage{geometry}

\def \qq {\boldsymbol{q}}

\def \ra {\rightarrow}

\def \E {\mathbb{E}}

\def \Re {\mathbb{R}}
\def \a {\alpha}
\def \be {\beta}

\newtheorem{example}{\bf Example}
\newtheorem{exm}[example]{\bf Example}

\newtheorem{definition}{\bf Definition}
\newtheorem{defn}[definition]{\bf Definition}

	\newtheorem{theorem}{\bf Theorem}
	
	\newtheorem{prop}{\bf Proposition}
	\newtheorem{lem}[theorem]{\bf Lemma}
	\newtheorem{as}{\bf Assumption}

\renewcommand{\epsilon}{\varepsilon}

\begin{document}

\begin{frontmatter}

\title{On Lock-down Control of a Pandemic Model}
\runtitle{Pandemic control}

\begin{aug}
\author{\fnms{Paramahansa} 
	\snm{Pramanik}
	\ead[label=e1]{ppramanik@southalabama.edu}}

\runauthor{P. Pramanik}

\affiliation{University of South Alabama}

\address{Department of Mathematics and Statistics\\ University of South Alabama\\ Mobile, AL 36688 USA.\\email:ppramanik@southalabama.edu\\Phone:251-341-3098.}

\end{aug}
  
\begin{abstract}
In this paper a Feynman-type path integral control approach is used for a recursive formulation of a health objective function subject to a fatigue dynamics, a forward-looking stochastic multi-risk susceptible-infective-recovered (SIR) model with risk-group's Bayesian opinion dynamics towards vaccination against COVID-19. My main interest lies in solving a minimization of a policy-maker's social cost   which depends on some deterministic weight. I obtain an optimal lock-down intensity from a Wick-rotated Schr\"odinger-type equation which is analogous to a Hamiltonian-Jacobi-Bellman (HJB) equation. My formulation is based on path integral control and dynamic programming tools facilitates the analysis and permits the application of algorithm to obtain numerical solution for pandemic control model. Feynman path integral is a quantization method which uses the quantum Lagrangian function, while Schr\"odinger's quantization uses the Hamiltonian function. These two methods are believed to be equivalent but, this equivalence has not fully proved mathematically. As the complexity and memory requirements of grid-based partial differential equation (PDE) solvers increase exponentially as the dimension of the system increases, this method becomes impractical in the case with high dimensions. As an alternative path integral control solves a class a stochastic control problems with a Monte Carlo method for a HJB equation and this approach avoids the need of a global grid of the domain of the HJB equation. 
\end{abstract}



\end{frontmatter}

\section{Introduction}
 In current days we see ``locking downs" of economies as a strategy to reduce the spread of COVID-19 which already has claimed more than 999,790 lives in the United States and more than 6 millions across the globe. Multiple countries have started this strategy to all the sectors of their economies except some essential service sectors such as healthcare and public safety. Different States in the United States locked down during different time periods based on their infection rates and  extremely contagious transmission phase. Re-opening has been prompted by slowing down the infection rate and wanes public activities \citep{caulkins2021optimal}. Locking down an economy for a long time may impact severely in the sense that, people might not come outside their homes for socioeconomic activities. A possible reason might be they are too afraid to communicate in-person thinking about themselves getting infected by this virus. As a result, even if a store is open for business activities, it might face a reduction of customers and even a reduction of its own employees. This may affect its profit in the long-run. If it does not have enough inventories, the store might shut-down in the long-run. Therefore, a business can be shut down quickly, but it is hard to re-open as the government cannot fiat money to them to return to its previous level of employment \citep{caulkins2021optimal}. This might be a reason why Centers for Disease Control and Prevention (CDC) recommends a person infected with Omicron should isolate themselves for five days.
 
 Condition for shut-down is determined  when a healthcare cost function is minimized subject to a stochastic multi-risk Susceptible-Infectious-Recovered (SIR) model \citep{kermack1927}. Almost all mathematical models of transmission of infectious disease models come from SIR model. This is the main reason to use this model. A lot of studies regarding dynamic behavior of different epidemic models have been done \citep{beretta1995,ma2004,xiao2007,rao2014,ahamed2021}. The deterministic part of this stochastic SIR model consists saturated transmission rate which depends on the location of that person. If that person commutes to or stay in the urban area, then they might have interaction with more people than a person who lives in a rural area, which reflects to a higher chance of getting infected. Diffusion part of the SIR model is needed when a person living in the rural area visits a city because of some arbitrary needs and gets in touch with others. On the other hand, poor air quality causes respiratory illness, affects adversely to cardiovascular health and deteriorates life expectancy \citep{delfino2005,albrecht2021}. In a similar manner random factor from the environment such as sudden change in the air quality due to volcanic eruptions, storms, wildfires and floods can affect the air quality drastically and lead to a more vulnerable atmosphere. Preexisting health conditions like obesity, diabetes, hypertension, weak immune system and higher age put a person towards higher risk to get infected by COVID-19 \citep{richardson2020,albrecht2021}. 
 
  In this paper a Feynman-type path integral approach has been used for a recursive formulation of a health objective function with a stochastic fatigue dynamics, \emph{forward-looking} stochastic multi-risk SIR model and a Bayesian opinion network of a risk-group towards vaccination against COVID-19. My main interest lies in solving a minimization problem $\bf{H}_{\theta}$ which depends on a deterministic weight $\theta$ \citep{marcet2019}. A Wick-rotated Schr\"odinger type equation (i.e. a Fokker-Plank diffusion equation) is obtained which is an analogous to a HJB equation \citep{yeung2006} and a \emph{saddle-point functional equation} \citep{marcet2019}. My formulation is based on path integral control and dynamic programming tools facilitates the analysis and permits the application of algorithm to obtain numerical solution for this stochastic pandemic control model. Furthermore, $\mathbf{H}_{\theta}$ with given initial conditions, is labeled as a \emph{continuation problem} as its solution coincides with the solution from period $s$ on wards \citep{marcet2019}. A terminal condition of the policy maker's objective function makes it as a \emph{Lagrangian problem} \citep{intriligator2002}.
  
  \emph{Feynman path integral} is a quantization method which uses the quantum \emph{Lagrangian} function, while\\ Schr\"odinger's quantization uses the \emph{Hamiltonian} function \citep{fujiwara2017}. As this path integral approach provides a different view point from Schr\"odinger's quantization,it is very useful tool not only in quantum physics but also in engineering, biophysics, economics and finance \citep{kappen2005,anderson2011,yang2014path,fujiwara2017}. These two methods are believed to be equivalent but, this equivalence has not fully proved mathematically as the mathematical difficulties lie in the fact that the \emph{Feynman path integral} is not an integral by means of a countably additive measure \citep{johnson2000,fujiwara2017}. As the complexity and memory requirements of grid-based partial differential equation (PDE) solvers increase exponentially as the dimension of the system increases, this method becomes impractical in the case with high dimensions \citep{yang2014path}. As an alternative one can use a Monte Carlo scheme and this is the main idea of \emph{path integral control} \citep{kappen2005,theodorou2010,theodorou2011,morzfeld2015}. This \emph{path integral control} solves a class a stochastic control problems with a Monte Carlo method for a HJB equation and this approach avoids the need of a global grid of the domain of HJB equation \citep{yang2014path}. If the objective function is quadratic and the differential equations are linear, then solution is given in terms of a number of Ricatti equations which can be solved efficiently \citep{kappen2007b,pramanik2020motivation,pramanik2021,pramanik2021scoring}. Although incorporate randomness with its HJB equation is straight forward but difficulties come due to dimensionality when a numerical solution is calculated for both of deterministic or stochastic HJB \citep{kappen2007b}. General stochastic control problem is intractable to solve computationally as it requires an exponential amount of memory and computational time because, the state space needs to be discretized and hence, becomes exponentially large in the number of dimensions \citep{theodorou2010,theodorou2011,yang2014path}. Therefore, in order to calculate the expected values it is necessary to visit all states which leads to the summations of exponentially large sums \citep{kappen2007b,yang2014path,pramanik2021}.
 
  \cite{acemoglu2020} suggests that, more restrictive policies about social interaction with people with advanced age reduce the COVID-19 infection for the rest of the population. In \cite{acemoglu2020} the population is divided into three age groups: young (22-44), middle-aged (45-65), and advanced-aged ($65+$) where the only differences in interactions between these groups come from different lock-down policies. Then they applied a deterministic multi-risk SIR model in each group and suggested that  using a uniform lock-down policy for the policymakers targeting stricter lock-down policy to more advanced aged population, the fatality rate due to COVID-19 would be just above $1\%$ (where uniform policy leads to a $1.8\%$ fatality rate). Targeted policy reduces the economic damage from $24.3\%$ to $12.8\%$ of yearly gross domestic product (GDP) \citep{acemoglu2020}. Furthermore, when targeted policies such as changing in norms and laws segregating the young population from the older are imposed, fatalities and economic damages because of COVID-19 can be substantially low \citep{acemoglu2020}. 
 
 The solutions to the optimal ``locking down" problem are very complicated in the sense that, if an economy imposes a stricter policy for a long time, it would be able to reduce the infection rate at a very low level. On the other hand, if the lock-down is short then, the policy makers are softening the infection rate of COVID-19 from touching down the peak \citep{caulkins2021optimal}. Another important assumption is that, the information regarding spreading of COVID-19 transmission is incomplete and imperfect. Therefore, one might have multiple \emph{Skiba} points or multiple solutions and none of them are unique. Rigorous studies about Skiba points have been done in \cite{skiba1978,grass2012} and \cite{sethi2019}. Although there is a growing literature on COVID-19 and its socioeconomic impacts related to extended lock-down time, length of lock-down and the appropriate time to lock down have not been studied that much \citep{caulkins2021optimal}. Furthermore, I am using a new Feynman-type path integral approach which has an advantage over traditional Hamiltonian-Jacobi-Bellman (HJB) approach as the complexity and memory requirements of grid-based partial differential equation increases exponentially with the dimension of the system \citep{yang2014,pramanik2020optimization,pramanik2021}.
 
 One can transform a class of non-linear HJB equations into linear equations by doing a logarithmic transformation. This transformation stems back to the early days of quantum mechanics which was first used by Schr\"odinger to relate HJB equation to the Schrödinger equation \citep{kappen2007a}. Because of this linear feature, backward integration of HJB equation over time can be replaced by computing expectation values under a forward diffusion process which requires a stochastic integration over trajectories that can be described by a path integral \citep{kappen2007a,pramanik2019,pramanik2021thesis}. Furthermore, in more generalized case like Merton-Garman-Hamiltonian system, getting a solution through Pontryagin Maximum principle is impossible and Feynman path integral method gives a solution \citep{baaquie1997,pramanik2020,pramanik2021,pramanik2021reg}. Previous works using Feynman path integral method has been done in motor control theory by \cite{kappen2005}, \cite{theodorou2010} and \cite{theodorou2011}. Applications of Feynman path integral in finance has been discussed rigorously in \cite{belal2007}.  A key assumption to get HJB is that the feasible set of action is constrained by a set of state and control variables only which does not satisfy many economic problems with \emph{forward-looking} constraints, where the future actions are also in the feasible set of actions \citep{marcet2019}. In the presence of a \emph{Forward-looking} constraints, optimal plan does not satisfy Pontryagin's maximum principle \citep{yeung2006} and the standard form of the solution ceases to exist because, the choice of an action carries an implicit promise about a future action \citep{marcet2019}. The absence of a standard recursive \citep{ljungqvist2012} formulation complicates the dynamic control problem with high dimensions and fails to give a numerical solution of the system \citep{yang2014,marcet2019}.
 
 Another important context is the rate of spread of COVID-19 in a community. The question of immunity and susceptibility is critical to the statistical analysis of infectious disease like COVID-19. Under the assumption that everybody in a community is susceptible to this pandemic one may be led to think that it is mildly infectious \citep{becker2017}. On the other hand, if everybody who had previously acquired immunity, is able to escape infection during this pandemic, one should conclude that it is highly infectious. Furthermore, immunity status of individuals  assessed by the tests on blood, saliva or excreta samples, is another determinant about the intensity of the spread of this pandemic \citep{becker2017}. Therefore, we are using network graph analysis to determine the spread of the infection. Based on the groups I have classified the social network directed graph and determine the adjacency matrix without existence of a loop. Furthermore, an undirected  network graph leads to a symmetric adjacency matrix \citep{pramanik2016,hua2019,polansky2021}. The diagonal terms of this matrix is zero and the off-diagonal terms have different values based on their weight in relation to the other persons in a community. For example, I give higher value to parents, spouses and siblings of a person compared to a person in distant relationship because if our person of interest gets infected by COVID-19, their parents, spouses and siblings are the ones who would be in risk to get infected by the pandemic.
 
 Opinion towards taking the vaccine is another important factor to determine the spread of COVID-19. When the policymakers in the United States has decided to mandate vaccination in all the public sector employees, many people have gone for a protest and significant number of government employees take leave from their duties  which has affected negatively towards those sectors such as New York Fire and Chicago Police Departments. Main reasons are: people think Government mandate for vaccination is against the civil right and, religious beliefs respectively. As social networks are the results of individual opinions, consensus towards the opinions regarding COVID-19 vaccine mandate takes an important role to understand the formation of spreading of infection in it. Although a lot of theoretical works on social networks have been done \citep{jackson2010,goyal2012,sheng2020}, work on effects of personal opinions towards the vaccine mandate on influencing of the spread of this disease is insignificant. \cite{sheng2020} formalizes network as simultaneous-move game, where social links based on decisions are based on utility externalities from indirect friends and proposes a computationally feasible partial identification approach for large social networks. The statistical analysis of network formation goes dates back to the seminal work by \cite{erdos1959} where a random graph is based on independent links with a fixed probability \citep{sheng2020}. Beyond Erd\"os-R\'enyi model, many methods have been designed to simulate graphs with  characteristics like degree distributions, small world, and Markov type properties \citep{polansky2021,pramanik2021consensus}. 
 
 Following is the structure of this paper. Beginning part of Section 2 discuss about about different COVID-19 spread and the definition of lock-down intensity. Section 2.1 talks about different stochastic dynamics needed for my analysis and their properties, Section 2.2 discuss about Bayesian opinion dynamics of a risk-group towards vaccination against COVID-19 and Section 2.3 discuss about the objective function of a policy maker. Theorem \ref{t0} in Section 3 is the main result of the paper. A closed form solution of lock-down intensity is calculated at the end of section 3 and finally, Section 4 discuss about the conclusion and future research of this context. 
  
 \section{Formulation of a Pandemic Model}
 
In this section, I provide the construction of a stochastic SIR model, fatigue dynamics, infection rate dynamics, opinion dynamics against  COVID-19 vaccination with a dynamic social cost as the objective function. Furthermore, I discuss how the stochastic programming method can be used to formulate a recursive formulation of a large class of pandemic control models with forward-looking stochastic dynamics.

\cite{acemoglu2020} considers three age groups young (22-44 years), middle-aged (45-65 years) and advanced-aged (65+ years). One can construct $K$ total number age-groups based on a group's vulnerability to COVID-19.  I assume equal group sizes for simplicity. For finite and continuous time $s\in[0,t]$ define a group vulnerable to COVID-19 is  $k$ such that, $k=1,2,...,K$ with $N_k$ be the initial population of an economy. Furthermore, I determine $K$ large enough to ensure every agent in an age-group has homegenous behavior.  At time $s$, the age-group (I will use the term \emph{risk-group} instead of age-group because each group is vulnerable to COVID-19 at certain extent) $k$ is subdivided into those susceptible (S), those infected (I), those recovered (R) and those deceased (D), $$ S_k(s)+I_k(s)+R_k(s)+D_k(s)=N_k.$$ Individuals in risk-group $k$ move from susceptible to infected, then either recover or pass away as well as groups also interact among themselves. 
\begin{figure}[H]
	\centering
	\begin{subfigure}{.6\textwidth}
		\centering
		\includegraphics[width=.6\linewidth]{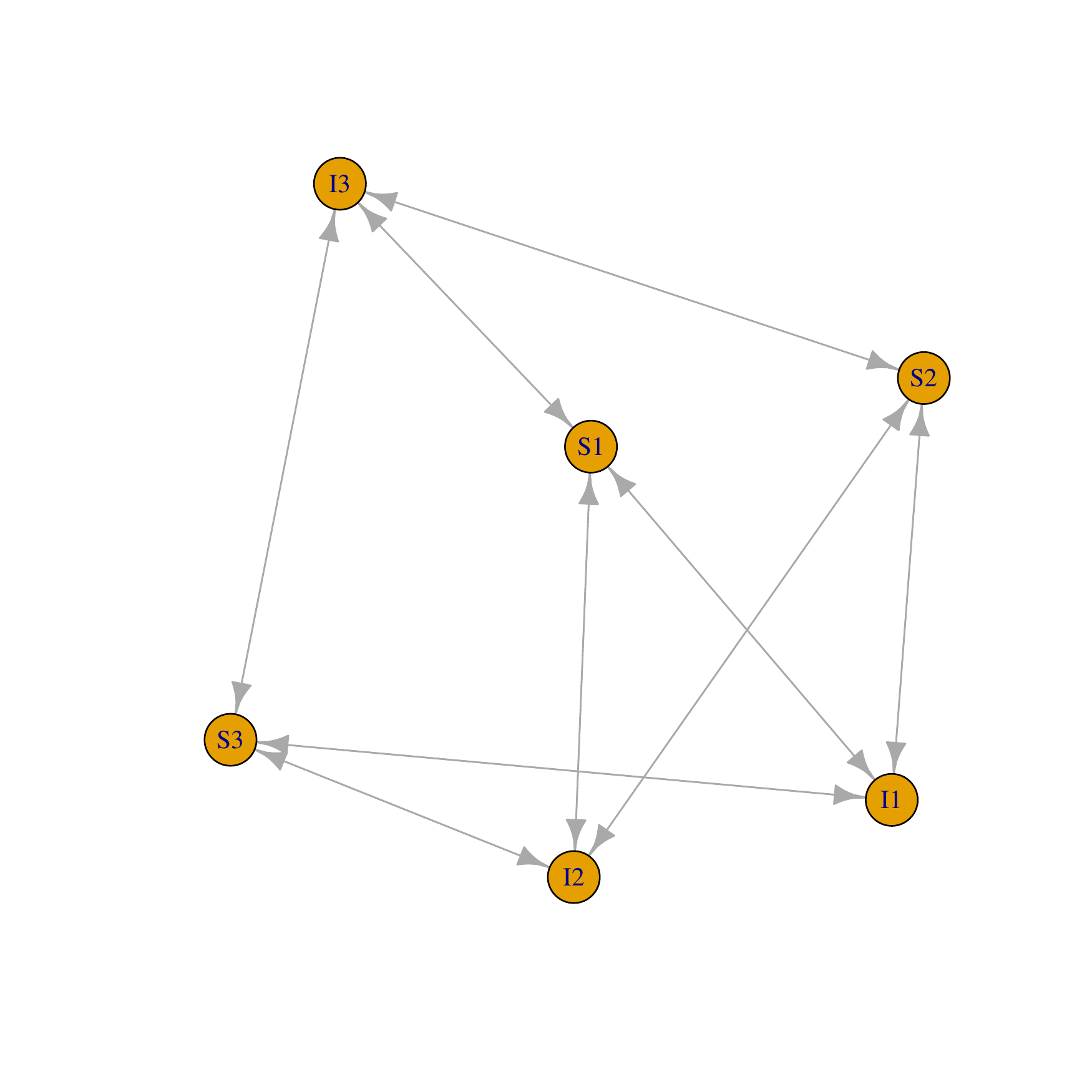}
		\caption{Connectivity between S and I among three risk groups.}
		\label{fig:sub1}
	\end{subfigure}
	\begin{subfigure}{.6\textwidth}
		\centering
		\includegraphics[width=.6\linewidth]{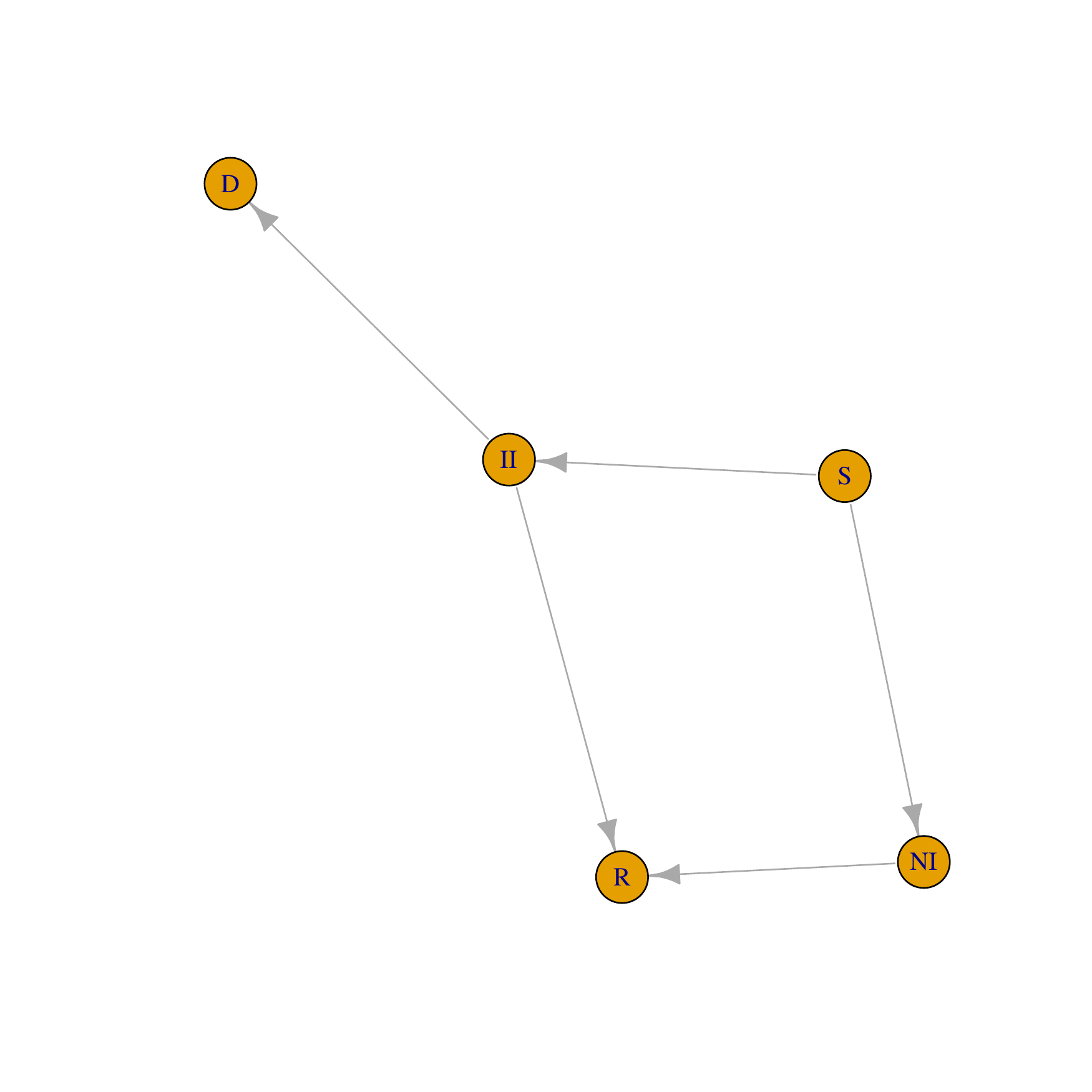}
		\caption{SIR for a single risk group.}
		\label{fig:sub2}
	\end{subfigure}
	\caption{Left panel represents the connectivity between Susceptibility (S) and Infection (I) among three risk-groups (i.e. young, middle-aged and old) while the right panel represents the state of an individual where NI represents a person is infected and under Non-ICU treatment while II indicates an individual is infected and is under ICU care. }
	\label{fig:test}
\end{figure}
In Figure \ref{fig:test} one can see how the state of an individual moves among the risk groups. Furthermore, the virus spreads exponentially. Therefore, the COVID-19 transmission follows a dynamic \emph{Barabasi-Albert} model where each new node is connected with existing nodes with a probability proportional to the number of links that the existing nodes already have \citep{barabasi1999}. 
\begin{figure}[H]
	\centering
	\begin{subfigure}{.7\textwidth}
		\centering
		\includegraphics[width=.7\linewidth]{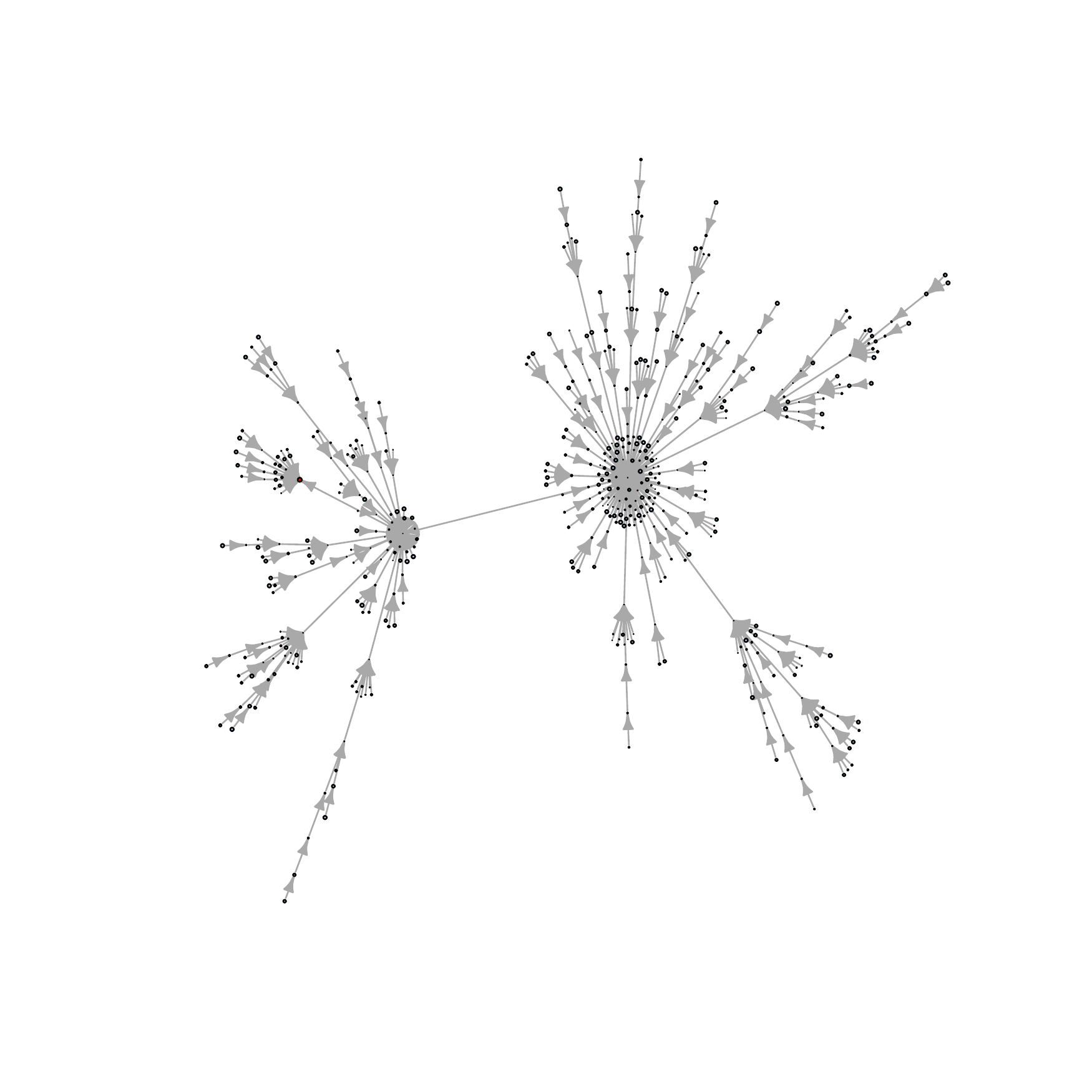}
		\end{subfigure}
	\begin{subfigure}{.7\textwidth}
		\centering
		\includegraphics[width=.7\linewidth]{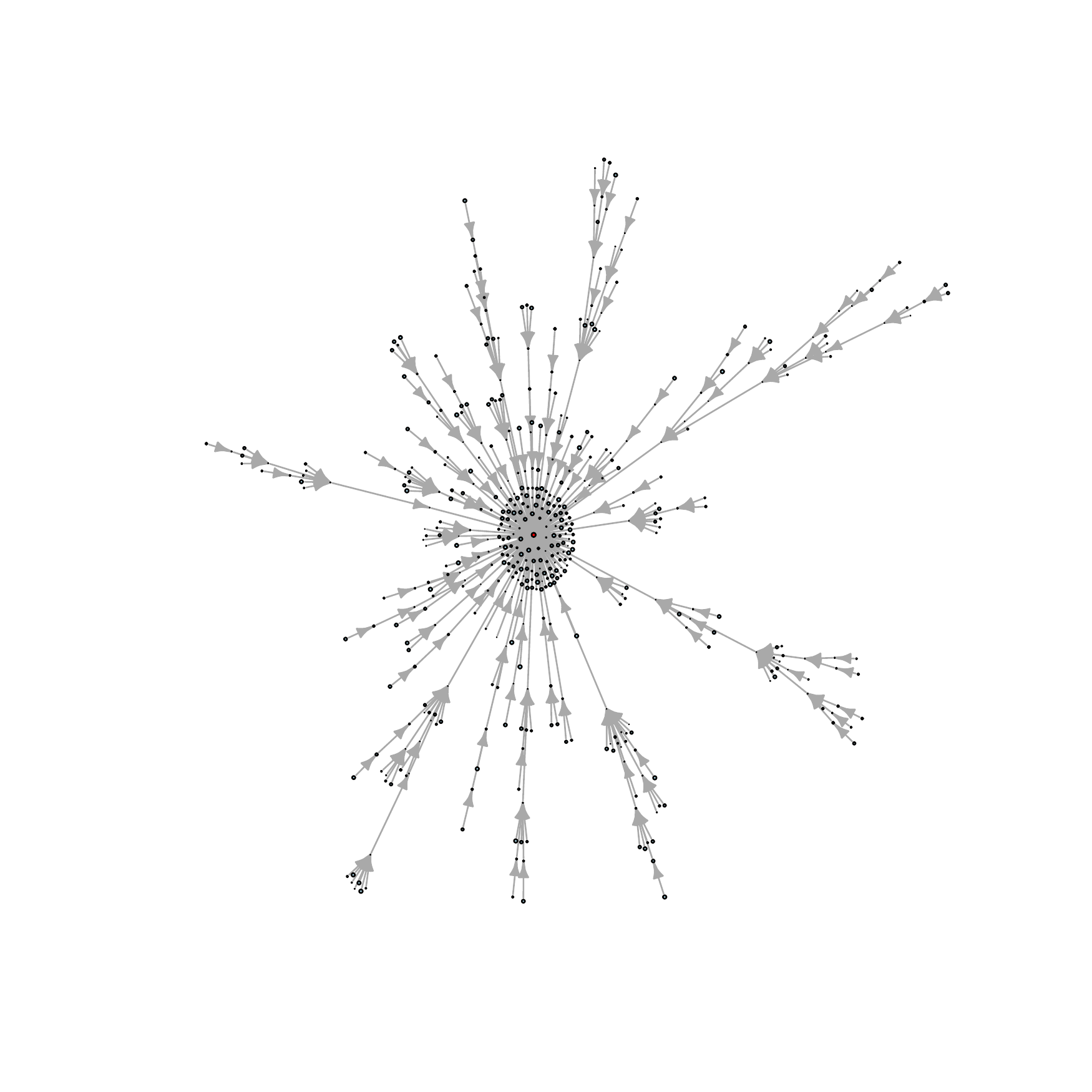}
		\end{subfigure}
	\caption{Two realizations of COVID-19 spread according to Barabasi-Albert model with $500$ vertices. }
	\label{fig:0}
\end{figure}
In Figure \ref{fig:0}, I construct two realizations of random COVID-19 spread  where the probability of each node depends on a person's immunity level. 

As lockdown and social distancing reduce interaction among people, I will treat ``lockdown" as a policy. Let for risk-group $k$, $L_k^b(s)$ is the total number of people willing to work before the pandemic and $L_k^a(s,c)$ is the total number of people willing to work during pandemic which is a function of lockdown fatigue (due to COVID-19 deaths in $k^{th}$ risk-group) is denoted by $c(s)$. Suppose, $d_1,d_2,\in(0,1)^2$ are the factors representing the proportions of  $L_k^b(s)$ and $L_k^a(s,c)$ respectively who are actually working. Define a new variable $e_k(s)=\frac{d_2L_k^a(s,c)}{d_1L_k^b(s)}$ as actual number of people working during pandemic as a proportion of those who are suppose to work without presence of COVID-19. At the very early stages as people have little knowledge about COVID therefore, $e_k(s)>1$. Furthermore, due to discoveries of vaccines and the incidence of the disease for more than a year, people's opinion against vaccination might lead indifference in behavior towards going to work or not. Therefore, $e_k(s)\downarrow 1$. In this case, policy makers come to place to restrict employment such that $e_k(s)\in (0,1)$. Thus, under policy-maker's intervention $e_k(s)=\frac{d_0d_2L_k^a(s,c)}{d_1L_k^b(s)}$ where, $d_0\in[0,1]$ is the parameter which is predetermined by the policymakers to restrict employment during pandemic. On the other hand, if the policy makers think an emergence of a  new variant of COVID-19 is random they fix $d_0=1$ and let the economy move on its way. For finite, continuous time $s\in[0,t]$ the ratio $e_k(0)=1$ and $e_k(t)\in[0,1]$ is based on the condition of pandemic. Hence, $\frac{\partial e_k}{\partial s}=u_k(s)$ represents the intensity of allowed employment and I use it as the stochastic control variable.

\subsection{Stochastic SIR Model}
Following \cite{caulkins2021optimal} I assume a state variable $z_k(s)$ capturing a ``lockdown fatigue" through a stochastic accumulation dynamics determined by COVID-19 related unemployment rate for risk-group $k$ is $[1-e_k(s)]$. The stochastic fatigue dynamics is given by
\begin{equation}\label{0}
dz_k(s)=[\kappa_0\{1-e_k(s)\}-\kappa_1z_k(s)p(\eta_{k_i},s)]ds+\sigma_0^k[z_k(s)-z_k^*]dB_0^k(s),
\end{equation}
where $\kappa_0$ indicates the rate of fatigue accumulation, $\kappa_1$ is the rate of exponential decay, 
\[
p(\eta_{k_i},s)=\left[\eta_{k_i}(s)\right]\left[\sum_{k_j=1}^{J-1}\eta_{k_j}(s)\right]
\]
 denotes the probability that a link of the new node connects to Barabasi-Albert node $k_i$ depends on the degree $\eta_{k_i}$ at time $s$ \citep{barabasi1999}, $\sigma_0^k$ is the diffusion coefficient,  $z_k^*$  is equilibrium value of $z_k$ and $B_0^k$ is a 1-dimensional Brownian motion. Under the absence of diffusion component  and under extreme lockdown (i.e., $e_k(s)=0$) this state variable takes its maximum value $Z_{\max}=\kappa_0/\left[\kappa_1p(\eta_{k_i},s)\right]$.

\begin{as}\label{as0}
	For $t>0$, let $\hat{\mu}(s,e_k,p,z_k):[0,t]\times [0,1]^2\times \mathbb{R} \ra\mathbb{R}$ and ${\sigma_0^k(z_k)}: \mathbb{R} \ra\mathbb{R}$ be some measurable function and, for some positive constant $K_1$,  $z_k\in\mathbb{R}$ we have linear growth as
	\[
	|\hat{{\mu}}(s,e_k,p,z_k)|+
	|{\sigma_0^k(z_k)}|\leq 
	K_1(1+|z|),
	\]
	such that, there exists another positive, finite, constant $K_2$ and for a different lockdown fatigue state variable $\tilde z_k$ 
	 such that the Lipschitz condition,
	\[
	|\hat{{\mu}}(s,e_k,p,z_k)-\hat{{\mu}}(s,e_k,p,\tilde z_k)|+|\sigma_0^k(z_k)-\sigma_0^k(\tilde z_k)|\leq K_2\ |z_k-\tilde z_k|,\notag
	\]
	$ \tilde z_k\in\mathbb{R}$ is satisfied and
	\[
	|\hat{{\mu}}(s,e_k,p,z_k)|^2+
	|\sigma_0^k(z_k)|^2\leq K_2^2
	(1+|\tilde z_k|^2).
	\]
\end{as}

\begin{as}\label{as1}
Assume $(\Omega,\mathcal F,\mathbb P)$ is the stochastic basis where the filtration $\{\mathcal F_s\}_{0\leq s\leq t}$ supports a 1-dimensional Brownian motion $B_0^k(s)=\\\{B_0^k(s)\}_{0\leq s\leq t}$. $\mathbb F^0$ is the collection of all $\mathbb R$-values progressively measurable process on $[0,t]\times\mathbb R$ and the subspaces are 
\[
\mathbb F^2:=\left\{ z_k\in\mathbb F^0;\ \E\int_0^t|z_k(s)|^2ds<\infty\right\}
\] 
and,
\[
\mathbb S^2:=\left\{Y_k\in\mathbb F^0;\ \E\sup_{0\leq s\leq t}|Y_k(s)|^2<\infty\right\},
\]
where $\Omega$ is the Borel $\sigma$-algebra and $\mathbb P$ is the probability measure \citep{carmona2016}. Furthermore, the 1-dimensional Brownian motion corresponding to lockdown fatigue for risk-group $k$ is defined as 
\[
B_0^k:=\left\{z_k\in\mathbb F^0;\ \sup_{0\leq s\leq t}|z_k(s)|<\infty;\ \mathbb P-a.s.\right\}.\]
\end{as}
 \begin{lem}\label{l0}
 Suppose the initial lockdown fatigue of $k^{th}$ risk group $z_k(0)\in\bf L^2$	is independent of Brownian motion $B_0^k(s)$ and the drift and the diffusion coefficients $\hat\mu(s,e_k,p,z_k)$ and $\sigma_0^k(z_k)$ respectively follow Assumptions \ref{as0} and \ref{as1} above. Then the lockdown fatigue dynamics in Equation (\ref{0}) is in space of the real valued process with filtration $\{\mathcal F_s\}_{0\leq s\leq t}$ and this space is denoted by $\mathbb{F}^2$. Furthermore, for some constant $c_0>0$, continuous time $s\in[0,t]$ and Lipschitz constants $\hat\mu$ and $\sigma_0^k$, the solution satisfies,
 \begin{equation}\label{1}
 \E \sup_{0\leq s\leq t}|z_k(s)|^2\leq c_0(1+\E|z_k(0)|^2)\exp{(c_0t)}.
 \end{equation}
 \end{lem}

\begin{proof}
	See in the Appendix.
\end{proof}

The foundation of pandemic model of our paper is stochastic Susceptibility-Infection-recovery (SIR) structure. Following \cite{acemoglu2020}, new infections are proportional to the number are proportional to the number of susceptible (S) and infected people (I)  of the initial population or $\beta SI$. Furthermore,  I assume that this infection rate $\be$ is subject to a random shocks \citep{lesniewski2020}, therefore, 
\begin{equation}\label{1.1}
d\be^k(s)=\left[\be_1^k+\be_2^kM\left\{e(s)^\theta+\frac{\kappa_0[z_k(s)]^\gamma}{\kappa_1p(\eta_{k_i},s)}\left(1-e_k(s)^\theta\right)\right\}\right]ds+\sigma_1^k(e_k(s),z_k(s))M dB_1^k(s),
\end{equation}
 where $\theta>1$ to make the function $\be^k(e_k,z_k)$ a convex function of $e_k$ (i.e., $\partial \beta^k/\partial e_k>0$ and $\partial^2 \beta^k/\partial e_k^2>0$), $\be_1^k,\be_2^k>0$ such that $\partial \beta/\partial z_k>0$, $\be_1$ is the minimum level of infection risk produced if only the essential activities are open, $\gamma\in(0,1)$ is the parameter which determines the degree of effectiveness of fatigue to spread infection, $M$ is fine particulate matter ($PM_{2.5}>12 \mu g/m^3$) which is an air pollutant and have significant contribution to degrade a person's health,  $\sigma_1^k(e_k(s),z_k(s))$ is a known diffusion coefficient infection dynamics and $dB_1^k(s)$ is one dimensional standard Brownian motion of $\be(e_k,z_k)$.  Therefore, in lack of presence of lockdowns and isolations, the new infection rate of group $k$ is \[
 S_k\frac{\sum_l\be^{kl}(s)I_l(s)}{\left[\sum_l \be^l(s)\left(S_l(s)+I_l(s)+R_l(s)\right)\right]^{2-\a}},
 \]
where $\be^{kl}$ are parameters which control infection rate between two infection groups $k$ and $l$ and, $\a\in[1,2]$ allows to control the returns of the scale matching \citep{acemoglu2020}. For steady state values $S_k^*$, $I_k^*$ and $R_k^*$ \citep{rao2014}, the risk-group $k$ has the SIR state dynamics as
\begin{align}\label{4}
dS_k(s) & = \biggr\{\eta N_k(s)-\be^k(e_k(s),z_k(s))\frac{S_k(s)I_k(s)}{\left[1+rI_k(s)\right]+\eta N_k(s)}-\tau S_k(s)\notag\\&\hspace{1cm}+\zeta R_k(s)\biggr\}ds+\sigma_2^k\left[S_k(s)-S_k^*\right]dB_2^k(s),\notag\\ dI_k(s)&=\left\{\be^k(e_k(s),z_k(s))\frac{S_k(s)I_k(s)}{\left[1+rI_k(s)\right]+\eta N_k(s)}-(\mu+\tau)I_k(s)\right\}ds\notag\\ &\hspace{1cm}+\sigma_3^k\left[I_k(s)-I_k^*\right]dB_3^k(s),\notag\\ dR_k(s) &=\left\{\mu I_k(s)-[\tau +\zeta]e_k(s) R_k(s)\right\}ds+\sigma_4^k\left[R_k(s)-R_k^*\right]dB_4^k(s),
\end{align}
where $\eta$ is birth rate, $1/\left[1+rI(s)\right]$ is a measure of inhibition effect from behavioral change of a susceptible individual in group $k$, $\tau$ is the natural death rate, $\zeta$ is the rate at which recovered person loses immunity and returns to the susceptible class and $\mu$ is the natural recovery rate of the infected individuals in risk-group $k$. $\sigma_2^k$, $\sigma_3^k$ and $\sigma_4^k$ are assumed to be real constants and are defined as the intensity of stochastic environment and, $B_2^k(s)$, $B_3^k(s)$ and $B_4^k(s)$ are standard one-dimensional Brownian motions \citep{rao2014}. It is important to note that in the dynamic systems (\ref{4}) is a very general case of SIR model.

For a complete  probability space  $(\Omega,\mathcal F,\mathbb P)$ with filtration starting from\\ $\{\mathcal F_s\}_{0\leq s\leq t}$, satisfying Assumptions \ref{as0} and \ref{as1}. Let 
\begin{equation*}
\mathbf Z_k(s)=\left[z_k(s),S_k(s),I_k(s),R_k(S)\right]\triangleq [h_1(s),h_2(s),h_3(s),h_4(s)],
\end{equation*}
where the norm $|\mathbf Z_k(s)|=\sqrt{z_k^2(s)+S_k^2(s)+I_k^2(s)+R_k^2(s)}$. Suppose, $C^{2,1}(\mathbb R^4\times(0,\infty),\mathbb R_+)$ be a family of all nonnegative functions $\mathfrak W(s,\mathbf Z_k)$ defined on $\mathbb R^4\times (0,\infty)$ so that they are continuously twicely differentiable in $\mathbf Z_k$ and once in $s$. Consider a differential operator $\mathbf D$ associated with 4-dimensional stochastic differential equation for risk-group $k$
\begin{equation}
d\mathbf Z_k(s)=\bm \mu_k(s,\mathbf Z_k)ds+\bm\sigma_k(s,\mathbf Z_k) d\mathbf B(s),
\end{equation}
such that
\begin{equation*}
\mathbf D=\frac{\partial}{\partial s}+\sum_{j=1}^4\mu_{k_j}(s,\mathbf Z_k)\frac{\partial}{\partial Z_{k_j}}+\frac{1}{2}\sum_{j=1}^4\sum_{j'=1}^4\left[\left[\bm\sigma_k^T(s,\mathbf Z_k)\bm\sigma_k(s,\mathbf Z_k)\right]_{jj'}\frac{\partial^2}{\partial Z_{k_j}Z_{k_{j'}}}\right],
\end{equation*}
where
\[\bm\mu_k=\begin{bmatrix}
\kappa_0(1-e_k)-\kappa_1z_kp(\eta_{k_i})\\
\eta N_k-\be^k(e_k,z_k)\frac{S_kI_k}{(1+rI_k)+\eta N_k}-\tau S_k+\zeta R_k\\
\be^k(e_k,z_k)\frac{S_kI_k}{\left[1+rI_k\right]+\eta N_k}-(\mu+\tau)I_k\\
\mu I_k-(\tau +\zeta)e_k R_k
\end{bmatrix}\]
and,
\[\bm\sigma_k=
\begin{bmatrix}
\sigma_0^k(z_k-z_k^*) & 0 &0 &0\\
0& \sigma_2^k(S_k-S_k^*) & 0&0\\
0& 0& \sigma_3^k(I_k-I_k^*)& 0\\
0& 0& 0&\sigma_4^k(R_k-R_k^*)
\end{bmatrix}.\]
Now let $\mathbf D$ acts on function $\mathfrak W\in C^{2,1}(\mathbb R^4\times(0,\infty);\mathbb R_+)$, such that
\begin{equation*}
\mathbf{D}\mathfrak W(s,\mathbf Z_k)=\frac{\partial}{\partial s}\mathfrak W(s,\mathbf Z_k)	+\frac{\partial}{\partial\mathbf Z_k}\mathfrak W(s,\mathbf Z_k)+\frac{1}{2}\text{$trace$}\left\{\bm\sigma_k^T(s,\mathbf Z_k)\left[\frac{\partial^2}{\partial Z_k^T \partial Z_k} \mathfrak W(s,\mathbf Z_k)\right]\bm\sigma_k(s,\mathbf Z_k)\right\},	
\end{equation*}	
where T represents a transposition of a matrix.

\begin{prop}\label{p0}
For any given set of initial values of risk-group $k$, $\{z_k(0),S_k(0),I_k(0),R_k(0)\}\in\mathbb R^4$ with Assumptions \ref{as0} and \ref{as1} there exists a unique solution $\{z_k(s),S_k(s),I_k(s),R_k(s)\}$ on $s\in[0,t]$ and will remain in $\mathbb R^4$ under incomplete and perfect information, where $B^k=B_0^k=B_2^k=B_3^k=B_4^k$.
\end{prop}

\begin{proof}
See in the Appendix.
\end{proof}

For theoretical purpose I rewrite theses equations as 
\begin{eqnarray}\label{5}
dS_k(s)&=& \mu_1(s,e_k,z_k,S_k,I_k,R_k)ds+\sigma_5^k(S_k)dB_2^k\notag,\\dI_k(s)&=& \mu_2(s,e_k,z_k,S_k,I_k)ds+\sigma_6^k(I_k)dB_3^k\notag,\\dR_k(s)&=& \mu_3(s,e_k,I_k,R_k)ds+\sigma_7^k(R_k)dB_4^k.
\end{eqnarray}
Furthermore, it is assumed to be the System (\ref{5}) follows Assumptions \ref{as0} and \ref{as1}.

\subsection{Opinion Dynamics of a risk-group $k$ towards vaccination against COVID-19}

This section will discuss about the spread of $k^{th}$ risk-group's opinion towards vaccination against  COVID-19 in the society. In the previous section I assume each risk-group is constructed such a way that each agent in that group has homogeneous opinions. Heterogeneous opinions need to be addressed by a multi-layer social-network which would be an interesting topic for future research and currently is beyond the scope of this paper. As there are $N_k$ agents in each of the $K$ risk-groups therefore, total population is $KN_k=N<\infty$. I assume that all risk-groups are connected to each other via an exogenous, directed network represented by graph $\mathcal G\subseteq N\times N$ which also represents how one risk-group spreads its beliefs about vaccination against COVID-19 to other risk-groups. For example, If risk-group $k$ gives its opinion to risk-group $l$, then I write $k\ra l$ or $(k,l)\in\mathcal G$. Furthermore, if risk-group $l$ gets different opinion about COVID-19 vaccination from risk-group $k$ more often then, $k$ and $l$ are \emph{group-neighbors} $\mathcal N_k(\mathcal G)$ \citep{board2021}. As COVID-19 is known less than two years to us, people have incomplete information about this pandemic and this leads to an incomplete information about the social network under COVID-19. This information is captured by finite signals $\chi_k\in X_k$ and a joint prior distributions over networks and signal profiles $\varrho(\mathcal G,\chi_k)$ \citep{board2021}. Now a random network $G=(N,X,\varrho)$. Consider following four cases:
\begin{itemize}
	\item Deterministic social network $\mathcal G$. Following \cite{board2021} signal spaces about the opinion of COVID-19 are assumed to be degenerate, $|X_k|=1$, and the prior $\varrho$ assigns probability $1$ to $\mathcal G$. Although complete information eases the situation, this is rare in current COVID-19 situation. As this pandemic is new, even policy makers do not have complete information. For example, at the middle of 2021 policymakers (such as Centers for Disease Control and Prevention (CDC)) announced that fully vaccinated people are completely safe against this pandemic. Now because of Omicron variant above $350,000$ people are infected daily by January 2022. As a result, people lose trust on policy-makers and make their opinions based on their beliefs and faiths. This makes the learning dynamics about COVID-19 extremely complicated. This motivates to study random opinion network about pandemic with incomplete information.
	\item Directed opinion network with finite types $\gamma\in \Gamma$ where, for a individual risk-group $k$, first I independently draw a finite type $\gamma\in\Gamma$ assuming any distribution with full support. After choosing $k^{th}$ risk-group's opinion types $\gamma$ against COVID-19 vaccination that risk-group randomly stubs each type $\gamma'$. Then during communication, type $\gamma'$ randomly stubs to type $\gamma'$ individual risk-groups. Now the individual risk-group knows total number of outlinks of each type in the sense that, what are their group-neighbor's stand towards COVID-19 vaccination. The outlink at time $s$ is denoted  as a vector $d=(s,d_{\gamma'})_{\gamma'}\in\mathbb N^{\gamma'}$ which is also realization of more generalized random vector $\mathcal D_\gamma=(s,\mathcal D_{\gamma,\gamma'})_{\gamma'}$ with expectation at time $s$ is $\E_s[\mathcal D_{\gamma,\gamma'}]$ where $\mathcal D=(s,\mathcal D_{\gamma,\gamma'})_{\gamma,\gamma'}$ is a \emph{time dependent or dynamic degree distribution}.
	\item Indirected opinion spread network with binary links and triangles. Following \cite{board2021} $k^{th}$ individual risk-group's spreading their opinions about vaccination against COVID-19 might have $\hat d$ binary stubs and $\tilde d$ pairs of triangles. 
	\begin{figure}[H]
		\centering
		\begin{subfigure}{.9\textwidth}
			\centering
			\includegraphics[width=.7\linewidth]{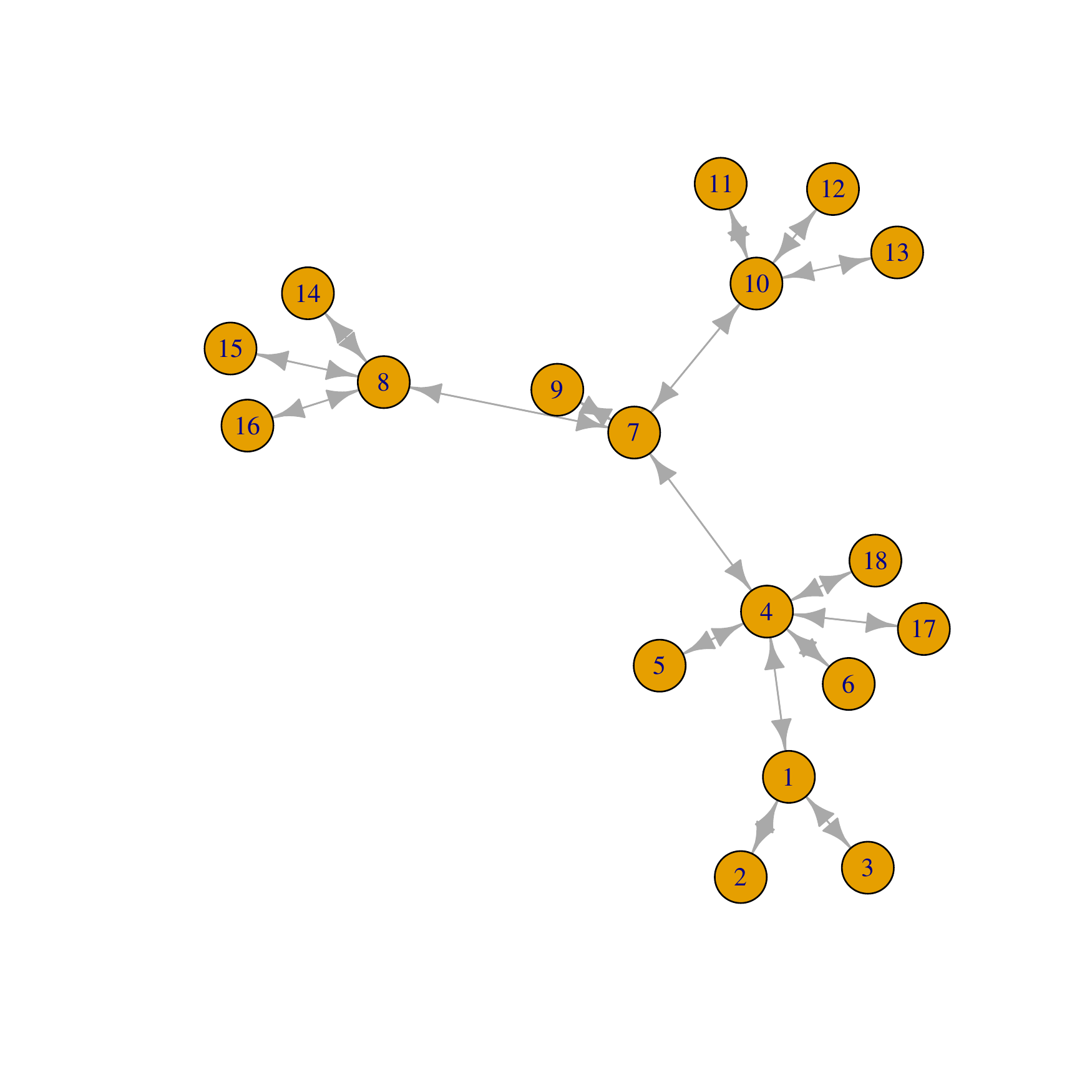}
			\caption{Binary stub where $\mathcal D$=1 for individual risk-groups \{2,3,5,6,9,11,12,13,14,15,16,17,18\}, $\mathcal D$=3 for individual risk-group 1, $\mathcal D$=4 for individual risk-groups \{7,8,10\} and $\mathcal D$=6 for individual risk-group 4.}
		\end{subfigure}
		\begin{subfigure}{.9\textwidth}
			\centering
			\includegraphics[width=.7\linewidth]{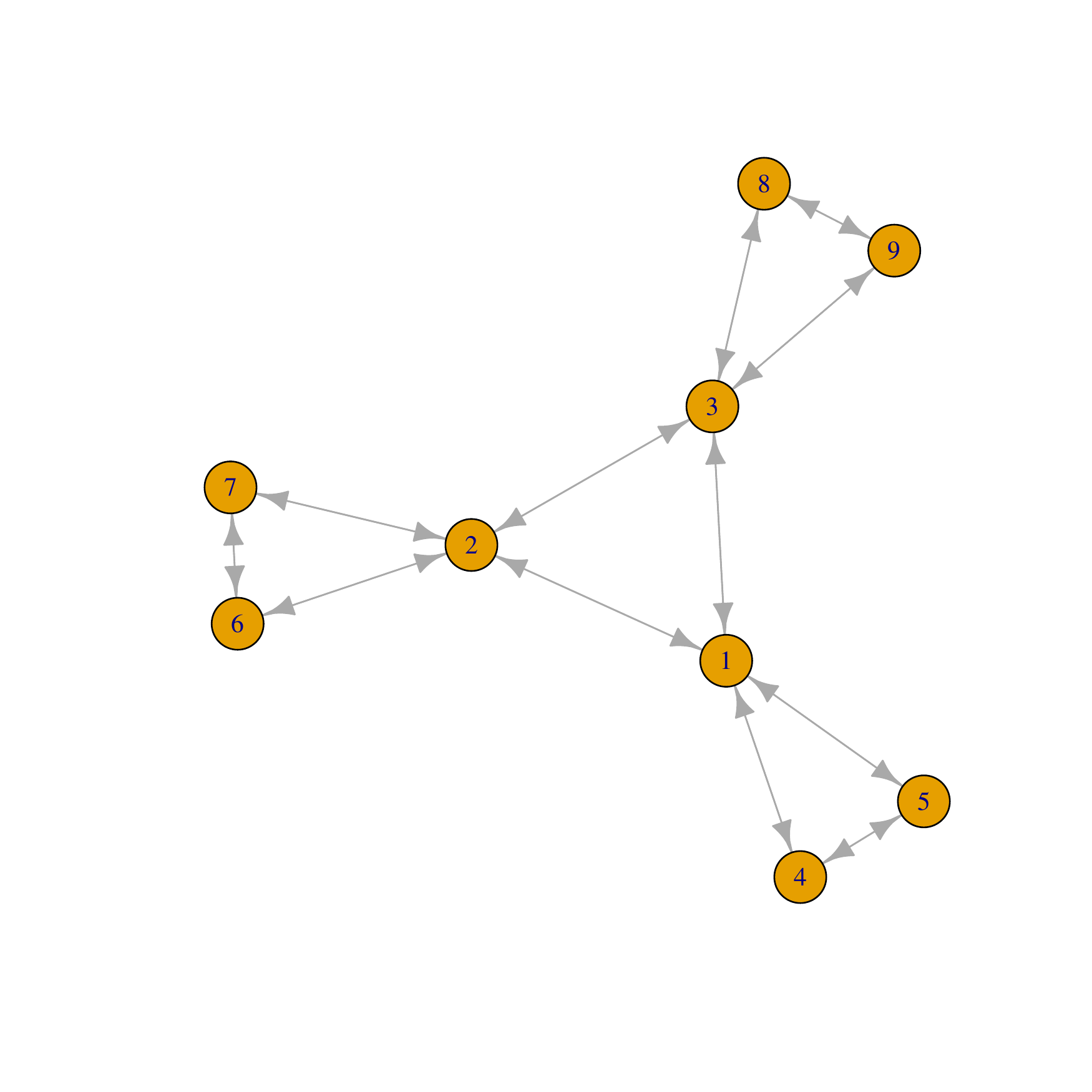}
			\caption{Here every individual risk-group has D-triangular stubs.}
		\end{subfigure}
		\caption{Two networks of individual risk-group $k$ such as binary and triangular stubs at time $s\in[0,t]$. }
		\label{fig:1}
	\end{figure}

From Figure \ref{fig:1} it is clear that $\hat d$ and $\tilde d$ are the subset of the above graph. For example, if we consider individual risk-group 1, then from the first panel it has $\hat d=3$ and in the second panel the same risk-group has two triangular stubs. We further assume, every individual risk-group knows their total number of binary and triangular stabs. In the world of COVID opinion spreading, if one individual risk-group shares their opinions to another risk-group very close to it then, the network connection might be triangular. On the other hand if individual risk-group $k$ spreads its opinion to some stranger (i.e., another risk-group far from risk-group $k$'s opinions), it would be one time binary information transition.
\item Microscopic interaction among risk-groups. A kinetic model for opinion spread towards vaccination against COVID-19 \citep{cordier2005,toscani2006}. Let $\omega_k$ denotes opinion of individual risk-group $k$ and it varies continuously between $-1$ and $1$. Here $-1$ represents an individual risk-group $k$'s extremely negative opinion for getting vaccinated against COVID-19 where as $1$ stands for completely opposite extreme opinion for COVID-19 vaccination. Following \cite {toscani2006} I assume that  directed and indirected interactions cannot destroy the bounds, which corresponds to imply that extreme opinions cannot be crossed. 
\end{itemize}

At the beginning of the interaction risk-group $k$ seeks to learn about the severity of COVID-19 with its own belief $v_k\in\{L,H,\omega_k\}=\{0,1,[-1,1]\}$, where L stands for low severidy of the disease and H stands for high severity. At $s=0$ and for a fixed belief against getting vaccinated, all the risk-groups share a common prior $Pr(v=H|\omega_k)=p_0\in(0,1)$, independent of network $\mathcal G$ and signals $X_k$. As the pandemic spreads, individual risk-group $k$ develops the need of information about the disease and starts interacting at time $s_k\sim U[0,1]$ (the uniform distribution  where $s_k$ is time-quantile during the presence of the pandemic). Based on the handling of the pandemic of the group-neighbors risk-group $k$ updates its probabilities of beliefs about pandemic to $Pr(\omega_k^*)=p_k^*$, such that $ Pr(\omega_k^*=-1)=0$ and $Pr(\omega_k^*=1)=1$. In order to get information, risk-group $k$ incurs some cost $c_k\sim F[\underline{c},\bar c]$, where F is the distribution function with bounded density function f. risk-group $k$ only gets exposure to the pandemic iff $v_k=\{L,\omega_k\}$. If individual risk-group $k$ does care about the severity of the disease, it interacts with other risk-groups frequently and transmits COVID-19. Interaction times $s_k$ and the cost of disease information $c_k$ are private information, independent within individual risk-groups in $S_k, I_k$ and $R_k$. 
 
If individual risk-group $k$ finds $v_k=\{L,\omega_k\}$ and does not mind to interact with other risk-groups, its utility becomes 1. If risk-group $k$ finds $v_k=\{H,\omega_k\}$ then, it is reluctant to interact with other risk-groups. In this case there are two possibilities, if unknowingly risk-group $k$ gets infected by the virus, its utility becomes $0$ and furthermore, if individual risk-group $k$ gets infected knowingly, its utility goes down to $-U$. Finally, if risk-group $k$ sees its group-neighbor gets infected by the virus but asymptotic, its posterior is $p_k=1$ and does not mind to interact. If risk-group $k$ gets infected by COVID-19 unknowingly, the posterior becomes $p_k\leq p_0$. Assume $U\geq p_0/(1-p_0)$, which leads to an adoption to the pandemic is a dominated strategy. Furthermore, if $(p_k-c_k)\geq 0$, then individual risk-group $k$ does not mind to interact with other risk-groups which might lead to get transmitted with the disease. On the other hand, if $(p_k-c_k)< 0$, then individual risk-group $k$ finds $v_k-\{H,\omega_k\}$ and tries to isolate from other risk-groups.

\begin{exm}
	Without loss of generality assume two independent risk-groups $k$ and $l$ who are interacted by a directed graph such that $k\ra l$. Before interaction, risk-group $k$ and $l$ have believes about COVID-19 vaccination as $\omega_k$ and $\omega_l$ respectively where $(\omega_k,\omega_l)\in\mathcal [-1,1]^2=\mathcal I^2$. Denote $Pr_s(L|\omega_k)$ as the probability of individual risk-group $k$'s willingness to contact with other risk-groups at time $s$ when it expects the severity of pandemic is less or L. Risk-group $k$ starts its communication at uniform time $s\in[0,t]$. As it is not rational for risk-group $k$ to interact with other risk-groups when $v_k$ is H, it is sufficient to keep track of the interaction probability conditional on $v_k=L$. Furthermore, as risk-group $k$ does not mind to interact as long as $c_k\leq p_0$ then $\partial [Pr_s(L|\omega_k)]/\partial k=Pr(k\ \text{ is indifferent to interact}\ |\omega_k)=F(p_0)$, which is independent of time. Furthermore, the interaction of opinions among risk-groups $k$ and $l$  follow the stochastic dynamic systems represented by
	\begin{equation*}
	d\omega_k(s)=\left\{\omega_k(s)-\varkappa e_k(s) \mathcal Q(|\omega_k(s)|)\left[\omega_k(s)-\omega_l(s)\right]\right\}ds+\sigma_8^ke_k(s)\left[\omega_k(s)-\omega_l(s)\right]dB_5^k(s),
	\end{equation*}
	
\begin{equation*}
d\omega_l(s)=\left\{\omega_l(s)-\varkappa e_l(s) \mathcal Q(|\omega_l(s)|)\left[\omega_l(s)-\omega_k(s)\right]\right\}ds+\sigma_9^le_l(s)\left[\omega_l(s)-\omega_k(s)\right]dB_6^l(s), 
\end{equation*} 
where $\varkappa\in(0,1/2)$ is the compromise propensity, the function $\mathcal Q(.)\in [0,1]$ with $\partial \mathcal Q/\partial \omega_k\leq 0$ represents the local relevance of compromise \citep{toscani2006}. It is important to know  that, if $e_k(s)\downarrow 0$ then there is a huge unemployment in the economy which means the incidence of pandemic is very severe. Under this case a difference in opinion $(\omega_k-\omega_l)$ does not affect the dynamic system and every risk-group needs to follow the policymakers' protocols. On the other hand, if $e_k(s)\downarrow 1$ then, opinion difference takes a major role to explain  the above stochastic opinion dynamical systems. Finally $\sigma_8^k(s)$ and $\sigma_9^i(s)$ are the opinion diffusion coefficients with $B_5^k(s)$ and $B_6^l(s)$ as their corresponding Browninan motions.

As risk-group $k$ interacted first, as a second mover individual risk-group $l$ learns about the effect of pandemic from risk-group k. Furthermore, if risk-group $l$ notices that, risk-group $k$ does not mind interacting with other risk-groups, then k thinks the disease is not fatal and is not reluctant to interact with others and, vice versa. Therefore, individual risk-group $l$'s posterior probability that COVID-19 is not severe is 
\[
p\left[Pr_s(L|\omega_k(s))\right]=\frac{\left\{1-\left[Pr_s(L|\omega_k(s))\right]\right\}p_0}{\left[1-Pr_s(L|\omega_k(s))\right] p_0+(1-p_0)}.
\]
Individual risk-group $l$ does not mind to interact with other risk-groups if $c_l\leq p\left[Pr_s(L|\omega_k(s))\right]$. As $Pr_s(L|\omega_k(s))$ changes based on the infection rate of the community, individual risk-group $l$'s optimistic approach to do social contact continues but the pessimistic approach kicks in only if $Pr_s(L|\omega_k(s))$ is starting to decrease. Therefore, individual risk-group $l$'s tolerance rate is 
\begin{eqnarray*}
\frac{\partial [Pr_s(L|\omega_l)]}{\partial l}&=&1-Pr(l \text{ is reluctant to interact}\ |\omega_k)\\&=& 1- Pr(k \text{ is reluctant to interact}\ |\omega_k)\\& &\times Pr(l \text{ is indifferent to interact}|k \text{ is reluctant to interact} ,\omega_k)\\ &=& 1-\left[1-Pr_s(L|\omega_k(s))\right]\left[1-F\left(p\left[Pr_s(L|\omega_k(s))\right]\right)\right]\\&=:&\hat\Phi\left[Pr_s(L|\omega_k(s))\right].
\end{eqnarray*}
By denoting $Pr_s(L|\omega_k(s))=W_k(s)$ and considering the stochastic opinion dynamics I define a stochastic differential equation
\begin{align}\label{6}
dW_k(s)=\mu_4\left[\omega_k(s)-\varkappa e_k(s) \mathcal Q(|\omega_k(s)|)\left[\omega_k(s)-\omega_l(s)\right]\right]ds\notag\\+\sigma_{10}^k[e_k(s)\left[\omega_k(s)-\omega_l(s)\right]dB_7^k(s)].
\end{align}
Without loss of generality the Equation (\ref{6}) becomes,
\begin{align}\label{7}
dW_k(s)=\mu_4(s,e_k,\omega_k,\omega_l)ds+\sigma_{10}^k(s,e_k,\omega_k,\omega_l)dB_7^k(s),
\end{align}
all the symbols have their usual meanings.
\end{exm} 

Let $G=(N,X,\varrho)$ be a random network with signal profile $\varrho(\mathcal G,\chi_k)$. Like in the example above I assume individual risk-group $l$ does not mind interacting socially with probability $Pr_s(L|\omega_l)=W_l(s)$. As risk-group $l$ does not have any prior knowledge about COVID-19 transmission network, its decision strictly depends on the actions of other risk-groups' willingness to do so  in the community $\mathcal G$ with signals $\varrho$. Let $W_{l,\mathcal G,\varrho,\chi_l,\omega_l}(s)$ be a social interaction function for risk-group $l$ subject to $(\mathcal G,\chi_l,\omega_l)$ after expectation over other risk-groups' time of social interaction is $s_k$ with cost $c_k$. After taking expectation on $(\mathcal G,\chi_{-l},\omega_{-l})$, consider $$W_{l,\chi_{-l},\omega_{-l}}(s):=\sum_{\mathcal G,\chi_{-l},\omega_{-l}}\varrho(s,\mathcal G,\chi_{-l},\omega_{-l}|\chi_{l},\omega_{l})W_{l,\mathcal G,\chi_{-l},\omega_{-l}}(s)$$ 
be risk-group $l$'s interim social interaction function such that its signal is $\chi_l$ and its own opinions $\omega_l$. Risk-groups under Bayesian social network are willing to do social interaction if their group-neighbors are not reluctant to interact with others. Suppose, at least one of individual risk-group $l$'s neighbor has the interim social interaction function $$W_{l,\chi_{-l},\omega_{-l}}'(s):=\sum_{\mathcal G,\chi_{-l},\omega_{-l}}\varrho(s,\mathcal G,\chi_{-l},\omega_{-l}|\chi_{l},\omega_{l})W_{l,\mathcal G,\chi_{-l},\omega_{-l}}'(s),$$ such that $c_l\leq p_l$.
To get a proper expression of $W_{l,\mathcal G,\chi_{-l},\omega_{-l}}(s)$ assume individual risk-group $l$ first observes whether their group-neighbors are engaged in social interactions. If they interact then risk-group $l$ gets the information that the pandemic is not severe and makes $p_l=1$.  On the other hand, if risk-group $l$ finds out their neighbors are keeping social distancing then risk-group $l$ will try to get more information if their opinions against the COVID-19 vaccination are very strong such that $c_l\leq \bar c_{l,\chi_l,\omega_l}:=p_l$, where $\bar c_{l,\chi_l,\omega_l}$ is some arbitrary cut-off cost depending on $\omega_l$. If individual risk-group $l$ finds out that the transmission of the pandemic is very high, it will put $p_l=0$. Therefore,
 \begin{eqnarray}\label{8}
 \frac{dW_{l,\chi_{-l},\omega_{-l}}(s)}{dl}&=& 1-\left\{(1-W_{l,\mathcal G,\chi_{-l},\omega_{-l}}(s))(1-F(p(W_{l,\chi_{-l},\omega_{-l}}(s))))\right\}\notag\\&=:&\phi\left\{1-F(p(W_{l,\chi_{-l},\omega_{-l}}(s)))\right\}.
 \end{eqnarray}
 
 \begin{lem}\label{l1}
 	For individual risk-groups $k$ and $l$, the pair of social interaction functions $\left(W_{k,\chi_{-k},\omega_{-k}}(s),W_{l,\chi_{-l},\omega_{-l}}(s)\right)$ on space $F=(s,G,N,X,\varrho,\mathcal I)$ with conditional probabilities $Pr_s(H|\omega_k)=1$ and $Pr_s(H|\omega_l)=1$ in a same community. Then under non-intersecting graph $\mathcal G$, different opinions and for a function $h\in F$ we have total social interaction variation as
 	\begin{eqnarray*}
 	&  &\left|\left|\left(W_{k,\chi_{-k}\omega_{-k}}(s)-W_{l,\chi_{-l},\omega_{-l}}(s)\right)\right|\right|\\&=& \sup\left\{\left|\left(W_{k,\chi_{-k},\omega_{-k}}(s,h)-W_{l,\chi_{-l},\omega_{-l}}(s,h)\right)\right|\right\}\\ &=& 1-\sup_{\hat h\in \left(W_{k,\chi_{-k},\omega_{-k}}(s),W_{l,\chi_{-l},\omega_{-l}}(s)\right)} \hat h(F)\\ &=& 1-\inf\sum_{i=1}^I\left(W_{k,\chi_{-k},\omega_{-k}}(s,\mathcal G_i)\wedge W_{l,\chi_{-l},\omega_{-l}}(s,\mathcal G_i)\right),
 	\end{eqnarray*}
 where the infimum is taken over all finite resolutions of F into pairs of nonintersecting subgraphs $\mathcal G_i$ with $I>1$.
 \end{lem}

\begin{proof}
	See in the Appendix.
\end{proof}
Above Lemma \ref{l1} implies that if social interaction function has bigger network (i.e. $\mathcal G$) then \[
\left|\left|\left(W_{k,\chi_{-k}\omega_{-k}}(s)-W_{l,\chi_{-l},\omega_{-l}}(s)\right)\right|\right|
\]
 will be small and vice versa. Therefore, if individual risk-group $l$ observes higher proportion of its neighbors are doing social interactions, they will do so. Furthermore, norm of social interaction is always less than unity. Therefore, the most extreme opinions against COVID-19 vaccination do not exist in this model.

Suppose, $\mathfrak q=\{\mathfrak q_k\}_{k\in K}$ represents the states of individual risk-group $k$, where $ \mathfrak q_k\in\{\bar q,\hat q,\tilde q\}$. If $\bar q=\emptyset$ then risk-group $k$ does not enter the COVID-19 network. If $\mathfrak q_k=\hat q$ then risk-group $k$ has entered the network but reluctant to do social interactions and finally, if $\mathfrak q_k=\tilde q$, then risk-group $k$ is in the network and is not maintaining social distance. Under the last case, $Pr_s(L|\omega_k)\approx 1$. Let $\Omega_E=\{0,1\}^E$ be the relevant finite sample space, containing configurations that allocate zeros and ones to the edge of $\mathcal G$, where $E=$edge of finite pandemic network $\mathcal G$ \citep{grimmett1995}. Consider $\delta\in \Omega_E$ the following condition holds,
\[
\delta(E)=\begin{cases}
1,\text{if edge E is open}\\ 0, \text{otherwise}.
\end{cases}
\]
The random cluster measure on COVID-19 social network $G$ with signal $\varrho$ and state profile $\mathfrak q$ is a probability measure at time $s\in [0,t]$
\begin{equation*}
\phi_{G,\varrho,\mathfrak q}(s,\delta)=\frac{1}{\Upsilon_{G,\varrho,\mathfrak q}}\left\{\prod_{E\in\mathbb E}\varrho^{\delta(E)}(1-\varrho)^{1-\delta(E)}\right\}\mathfrak q^{\jmath(\delta)},
\end{equation*}
where $\jmath(\delta)$ is the total number of open components of $\delta$, $\mathbb E$ is the space of all edges of the graph $\mathcal G$ and, $\Upsilon_{G,\varrho,\mathfrak q}$ is a normalizing factor (or,  partition function ) such that,
\[
\Upsilon_{G,\varrho,\mathfrak q}=\sum_{\delta\in \Omega_E}\left\{\prod_{E\in\mathbb E}\varrho^{\delta(E)}(1-\varrho)^{1-\delta(E)}\right\}\mathfrak q^{\jmath(\delta)}.
\]
A partial ordering under $\Omega_E$ given by $\delta\leq\delta'$ iff $\delta(E)\leq\delta', \ \forall E\in\mathbb E$. A function $\mho:\Omega_E\ra\mathbb E$ is called increasing if $\mho(\delta)\leq\mho(\delta'),\ \forall \ \delta\leq\delta'$. $\mathcal A$ is an increasing event if its simple function $\mathbbm 1_{\mathcal A}$ is increasing. Furthermore, if $\iota$ be a probability measure and $W_{k,\chi_{-k}\omega_{-k}}(s)$ be a random response function then, $\iota\left[W_{k,\chi_{-k}\omega_{-k}}(s)\right]$ is the conditional expectation of $W_{k,\chi_{-k}\omega_{-k}}(s)$ under $\iota$ \citep{grimmett1995}. In pandemic social network if $\mho$ and $W_{k,\chi_{-k}\omega_{-k}}(s)$ are increasing on the sample space $\Omega_E$, then 
\begin{equation*}
\phi_{G,\varrho,\mathfrak q}[\mho,W_{k,\chi_{-k}\omega_{-k}}(s)]\geq \phi_{G,\varrho,\mathfrak q}(\mho)\times \phi_{G,\varrho,\mathfrak q}[W_{k,\chi_{-k}\omega_{-k}}(s)].
\end{equation*}
Above inequality is called as Fortuin–Kasteleyn–Ginibre (FKG) inequality \citep{grimmett1995} of pandemic social network. Let $\mathbb Z^\gamma$ be a $\gamma$-dimensional hyperbolic Lattice such that risk-groups (i.e. vertices) $y_1$ and $y_2$ both are in it. For $E\subseteq\mathbb E$, $\mathcal F_E^k$ is the $\sigma$-field  such that $\mathcal F=\mathcal F_E^k$ \citep{grimmett1995}. $\Lambda\subseteq \mathbb Z^\gamma$ is a box such that,
\[
\Lambda=\prod_{\gamma=1}^\Gamma[y_1^\gamma,y_2^\gamma], 
\]
where $[y_1^\gamma,y_2^\gamma]$ is defined as $[y_1^\gamma,y_2^\gamma]\cap\mathbb Z$. The reason behind choosing a finite box $\Lambda$ inside  $\mathbb Z^\gamma$ is under the presence of COVID-19 risk-groups are not able to move across regions. Furthermore, moving around the globe is much harder because different countries have different restriction measures, which leads risk-groups to stay at home. As after certain point of time the COVID-19 infections go down, risk-groups would do social interactions locally. On the other hand, if a COVID-19 restriction stays too long, risk-groups would reluctant to stay at home. In this paper I am ruling out this  scenario. The box $\Lambda$ generates a sub-social network of lattice $\mathbb L$ with risk-group $k$ with $S_k,I_k$ and $R_k$ combined as set $\Lambda$ with the set of network connections $\mathbb E_\Lambda$. Define the $\sigma$-field at time s outside the network of $\Lambda$ as $\mathcal F_\Lambda=\mathcal F_{\mathbb E\setminus\mathbb E_\Lambda}$ and $\mathcal F=\cap_\Lambda\mathcal F_\Lambda$  as outside $\sigma$-field.

\begin{definition}\label{d0}
	A probability distribution $\phi$ on $G=(N,X,\varrho)$ with filtration $\mathcal F$ is called a random opinion cluster towards COVID-19 for three states $\mathfrak q$ and signal profiles $\varrho$ if
	\begin{equation*}
	\phi(\mathcal A|\mathcal F_\Lambda)=\phi_{\Lambda,\varrho,\mathfrak q}(\mathcal A), \ \mbox{$\phi$-a.s., for every $\mathcal A\in\mathcal F$ and boxes $\Lambda$.}
	\end{equation*}
	We denote this set as $\Re_{\varrho,\mathfrak q}$.
\end{definition}

\begin{defn}\label{d1}
	A probability distribution $\phi$ on $G=(N,X,\varrho)$ with filtration $\mathcal F$ is called a limit random opinion cluster towards COVID-19 for three states $\mathfrak q$ and signal profiles $\varrho$ if $\exists\  \xi\in\Omega$ and an increasing sequence of opinion boxes $\{\Lambda_n\}_{n\geq 1}$ such that
	\[
	\phi_{\Lambda_n,\varrho,\mathfrak q}^\xi\ra\phi,\ \mbox{as}\ n\ra\infty,
	\]
	where $\Lambda_n\ra\mathbb Z^\gamma$ as $n\ra\infty$ \citep{grimmett1995}.
\end{defn}

Furthermore, if the structure of network in a box $\Lambda$ is same (i.e. $\phi_{\Lambda,\varrho,\mathfrak q}^k=\phi_{\Lambda,\varrho,\mathfrak q}^l$ ) then for risk-groups k and l in the society are in $\mathbb R_{\varrho,\mathfrak q}$ and following \cite{grimmett1995}  $|\mathbb R_{\varrho,\mathfrak q}|=1$.

\begin{prop}\label{p1}
	Let for any random network $G$ with the signal profile $\varrho$ and for $\mathfrak q=\mathfrak q_3=\{\bar q,\hat q,\tilde q\}$ and social interactions of risk-group $k$ as $W_{k,\chi_{-k},\omega_{k}}$ exists and definitions \ref{d0} and \ref{d1} holds. Then there exists a unique random opinion distribution.
\end{prop}
\begin{proof}
	See in the Appendix.
\end{proof}
Proposition \ref{p1} guarantees that if risk-group $k$ has imperfect and complete information then under $\mathfrak q_3$ the random network has a unique solution.

\subsection{Objective function}
So far I have discussed about the stochastic dynamic systems of fatigue ($z_k$), infection rate ($\be^k$), multi-risk SIR ($S_k,I_k$ and $R_k$) and opinion of risk-group $k$ ($\omega_k$) with its probability conditioned on less severity as $W_k$. This section will discuss about the objective of the policy makers subject to the stochastic dynamics discussed above.

Let $\mathcal H_k(s)$ be the total number individuals of risk-group $k$ who need emergency care at time $s$. Hence, $\mathcal H_k(s)=\check{h}I_k(s)$, where $\check h\in(0,1)$ is some given proportionality constant available at time $s$ \citep{acemoglu2020}. Therefore, total number of people in $K$ risk-groups who need emergency care is $\mathcal H(s)=\sum_{k=1}^K\mathcal H_k(s)$. Following \cite{acemoglu2020} I assume that probability of death such that the person was under emergency care as $\varpi_k(s)=\varphi_k[\mathcal H(s)]$, for some given function $\varphi_k$. In this analysis a cost of death or value of life is included as $\breve\chi_k$ \citep{acemoglu2020}. By value of life I mean value of increasing the survival probabilities marginally due to COVID-19. In other words, one can think about the impact of death on a family in risk-group $k$ in terms of monetary loss and  emotional losses of that person's family as well as risk-group $k$. A policy maker considers this cost as non-pecuniary cost of death and is denoted by $\breve\chi_k\check h\varpi_k(s)I_k(s)$ as $\check h\varpi_k(s)I_k(s)$ is defined as the flow of death. 

I assume that the detection of a person infected by COVID-19 is imperfect as well as their isolation status. With out loss of generality assume $\tau_k$ be the constant probability that an infected person in risk-group $k$ does not need an emergency care and based on that person's $F(p_0)$-value risk-group $k$ would decide whether it will isolate that person or not. If $F(p_0)\downarrow 1$ then individual in risk-group $k$ will not be isolated with probability $\tau_kF(p_0)$ or simply $\tau_k$. On the other hand, if $F(p_0)\downarrow 0$, individual in risk-group $k$ will be isolated with probability $\tau_kF(p_0)$. Let $\hat\tau_k$ be the probability where an individual in risk-group $k$ is detected and need an emergency care for recovery. Hence, $F(p_0)$ is not as powerful as the case for those who do not need ICU care. Therefore, I restrict the upper limit of $F(p_0)$ as $\hat F_p<1/2$. This part is some extension of \cite{acemoglu2020} where individual opinion of risk-group $k$ was not considered. Therefore, the probability that a person is infected by COVID-19, detected and isolated in risk-group $k$ is 
\[
\check h\hat\tau_k\hat F_p+(1-\check h)\tau_k F(p_0).
\]
In the presence of Omicron, a completely vaccinated and boosted person in risk-group $k$ would have some probability $\tilde\tau_k$ to get infected by COVID-19 again. Therefore, I assume that the probability of a recovered person not to get infected by COVID-19 for risk-group $k$ is $(1-\tilde\tau_k)$. Due to imperfect testing assume a fraction $\breve\tau_k$ of recovered person in risk-group $k$ with probability $(1-\tilde\tau_k)$ are allowed to join the workforce freely. Remaining part of the recovered population is either not identified \citep{acemoglu2020} or because of the traumatic experience their $F(p_0)$ is very low and reluctant to join in the labor force. Therefore, the employment for somebody in $k^{th}$ risk-group at time $s$  is given by 
\begin{equation}\label{12}
\mathcal E_k(s)=e_k(s)\biggr\{S_k(s)+\left[1-\check h\hat\tau_k\hat F_p-(1-\check h)\tau_k F(p_0)\right]I_k(s)+(1-\breve\tau_k)\tilde\tau_kR_k(s)\biggr\}+\breve\tau_k(1-\tilde\tau_k)R_k(s).
\end{equation}
A policymaker has to control $\{e_k(s)\}_{k\in K}$ for all $s\in[0,t]$ where the dynamical system follows Equations (\ref{0}), (\ref{4}) and (\ref{6}). Planner's objective function is to minimize the expected present value of the social cost conditioned on the filtration $\mathbb F^0$ as 
\begin{multline}\label{13}
\mathbf H_\theta:\mathbf H_\theta^k(s,e_k,z_k,S_k,I_k,R_k,W_k)\\=\min_{\{e_k,z_k,S_k,I_k,R_k,W_k\}}\E_0\left\{\int_0^t\biggr[\exp\{-\rho s\}\sum_{k=1}^K\theta_kz_k(s)\left[N_k-\mathcal E_k(s)\right]+\breve\chi_k\check h\varpi_k(s)I_k(s)\biggr]ds\biggr|\mathbb F^0\right\},
\end{multline}
where $\theta_k>0$ is some known penalization constant, $\rho\in(0,1)$ is time independent discount rate and $\E_0$ is the conditional expectation at time 0 on the initial state variables $z_k(0),S_k(0),I_k(0),R_k(0)$ and $W_k(0)$ with filtration $\mathbb F^0$.
\begin{as}\label{as2}
	Following set of assumptions regarding the objective function is considered:
\begin{itemize}
	\item $\{\mathcal F_s\}$ takes the values from a set $\mathfrak Z\subset \mathbb R^{5K}$. $\{\mathcal F_s\}_{s=0}^t$ is an exogenous Markovian stochastic processes defined on the probability space $(\mathfrak Z_\infty,\mathbb F^0,\mathbb P)$.
	\item For all $\{e_k(s),z_k(s),S_k(s),I_k(s),R_k(s),W_k(s)\}$, there exists an optimal lock intensity $\{\overline e_k(s)\}_{s=0}^t$, with initial conditions $z_k(0),S_k(0),I_k(0),R_k(0)$ and $W_k(0)$, which satisfy the stochastic dynamics represented by the equations (\ref{0}), (\ref{4}) and (\ref{6}) for all continuous time $s\in[0,t]$.
	\item The function $ \exp\{-\rho s\}\sum_{k=1}^K\theta_kz_k(s)\left[N_k-\mathcal E_k(s)\right]+\breve\chi_k\check h\varpi_k(s)I_k(s)$ is uniformly bounded, continuous on both the state and control spaces and, for a given $\{e_k(s),z_k(s),S_k(s),I_k(s),R_k(s),W_k(s)\}$, they are $\mathbb F^0$-measurable.
	\item The function $ \exp\{-\rho s\}\sum_{k=1}^K\theta_kz_k(s)\left[N_k-\mathcal E_k(s)\right]+\breve\chi_k\check h\varpi_k(s)I_k(s)$ is strictly convex with respect to the state and the control variables.
	\item For all $\{e_k(s),z_k(s),S_k(s),I_k(s),R_k(s),W_k(s)\}$, there exists a $k$-interior lock intensity $\{\widetilde e_k(s)\}_{s=0}^t$, with initial conditions $z_k(0),S_k(0),I_k(0),R_k(0)$ and $W_k(0)$ satisfy Equations (\ref{0}), (\ref{4}) and (\ref{6}), such that
	\[
	\E_0\left\{\biggr[\exp\{-\rho s\}\sum_{k=1}^K\theta_kz_k(s)\left[N_k-\mathcal E_k(s)\right]+\breve\chi_k\check h\varpi_k(s)I_k(s)\biggr]\biggr|\mathbb F^0\right\}>0,
	\]
	and, for $k\neq l$
	\[
	\E_0\left\{\biggr[\exp\{-\rho s\}\sum_{k=1}^K\theta_k\tilde z_k(s)\left[N_k-{\widetilde{\mathcal E_k}}(s)\right]+\breve\chi_k\check h\varpi_k(s)\tilde I_k(s)\biggr]\biggr|\mathbb F^0\right\}\geq0.
	\]
	\item In addition to the above argument, there exists an $\varepsilon>0$ such that for all $\{e_k(s),z_k(s),S_k(s),I_k(s),R_k(s),W_k(s)\}$,
	\[
	\E_0\left\{\biggr[\exp\{-\rho s\}\sum_{k=1}^K\theta_k\tilde z_k(s)\left[N_k-{\widetilde{\mathcal E_k}}(s)\right]+\breve\chi_k\check h\varpi_k(s)\tilde I_k(s)\biggr]\biggr|\mathbb F^0\right\}\geq\varepsilon.
	\]
\end{itemize}
\end{as}

\begin{definition}\label{d2}
	For individual risk-group $k$ optimal state variables\\  $z_k^*(s),S_k^*(s),I_k^*(s),R_k^*(s)$ and, $W_k^*(s)$  and their continuous optimal lock intensity  $e_k^*(s)$ constitute a stochastic dynamic  equilibrium  such that for all $s\in[0,t]$ the conditional expectation of the objective function is
	\begin{multline*}
	\E_0\left\{\int_0^t\biggr[\exp\{-\rho s\}\sum_{k=1}^K\theta_kz_k^*(s)\left[N_k-\mathcal E_k^*(s)\right]+\breve\chi_k\check h\varpi_k(s)I_k^*(s)\biggr]ds\biggr|\mathbb F_*^0\right\}\\\leq \E_0\left\{\int_0^t\biggr[\exp\{-\rho s\}\sum_{k=1}^K\theta_kz_k(s)\left[N_k-\mathcal E_k(s)\right]+\breve\chi_k\check h\varpi_k(s)I_k(s)\biggr]ds\biggr|\mathbb F^0\right\},
	\end{multline*}
	with the dynamics explained in Equations (\ref{0}), (\ref{4}) and (\ref{6}), where $\mathbb F_*^0$ is the optimal filtration starting at time $0$ such that, $\mathbb F_*^0\subset\mathbb F^0$.
\end{definition}

\begin{definition}\label{d3}
Suppose, $z_k,S_k,I_k,R_k$ and $W_k$ are in a non-homogeneous Fellerian semigroup on continuous time interval $[0,t]$ in $\mathbb{R}^{6K}$. The infinitesimal generator $\mathfrak H$ of $\{z_k,S_k,I_k,R_k,W_k\}$ is defined by,
\begin{equation*}
\mathfrak H \mathbf H_\theta^k(e_k,z_k,S_k,I_k,R_k,W_k)=\lim_{s\downarrow 0}\frac{\E_s[\mathbf H_\theta^k(e_k,z_k,S_k,I_k,R_k,W_k)]-\mathbf H_\theta^k(e_k,\bar z_k,\bar S_k,\bar I_k,\bar R_k,\bar W_k)}{s},
\end{equation*}
for $\{z_k,S_k,I_k,R_k,W_k\}\in\mathbb{R}^{5K}$	where $\mathbf H_\theta^k:\mathbb{R}^{6K}\ra\mathbb{R}$ is a $C_0^2(\mathbb{R}^{6K})$ function,\\ $\{z_k,S_k,I_k,R_k,W_k\}$ has a compact support, and at $\{\bar z_k,\bar S_k,\bar I_k,\bar R_k,\bar W_k\}$ the limit exists where $\E_s$ represents individual risk-group $k$'s conditional expectation on states $\{z_k,S_k,I_k,R_k,W_k\}$ at continuous time $s$. Furthermore, if the above Fellerian semigroup is homogeneous over time, then $\mathfrak H \mathbf H_\theta^k$ is exactly equal to the Laplace operator.	
\end{definition}

As $\mathbf H_\theta^k$ is a $\mathbb F^0$-measurable function depending on $s$, there is a possibility that this function might have very large values and may be unstable. In order to stabilize the state variables $z_k,S_k,I_k,R_k,W_k$ I take the natural logarithmic transformation and define a characteristic like operator as in Definition \ref{d4}.

\begin{definition}\label{d4}
	For a Fellerian semigroup $\{z_k,S_k,I_k,R_k,W_k\}$ and for a small time interval $[s,s+\epsilon]$ with $\epsilon\downarrow 0$, define a characteristic-like operator where the process starts at $s$ is defined as 
	\begin{equation*}
	\hat{\mathfrak H}\mathbf H_\theta^k(e_k,z_k,S_k,I_k,R_k,W_k) =\lim_{\epsilon\downarrow 0}
	\frac{\log\E_s[\epsilon^2\ \mathbf H_\theta^k(e_k,z_k,S_k,I_k,R_k,W_k)]-\log[\epsilon^2\mathbf H_\theta^k(e_k,\bar z_k,\bar S_k,\bar I_k,\bar R_k,\bar W_k)]}{\log\E_s(\epsilon^2)},
	\end{equation*}
	for $\{z_k,S_k,I_k,R_k,W_k\}\in\mathbb{R}^{5K}$,	where $\mathbf H_\theta^k:\mathbb{R}^{5K}\ra\mathbb{R}$ is a $C_0^2(\mathbb{R}^{5K})$ function, $\E_s$ represents the conditional expectation of state variables $\{z_k,S_k,I_k,R_k,W_k\}$ at time $s$,  for $\epsilon>0$ and a fixed $\mathbf H_\theta^k$ the sets of all open balls of the form $B_\epsilon\left(\mathbf H_\theta^k\right)$ contained in $\mathcal{B}$ (set of all open balls) and as $\epsilon\downarrow 0$ then $\log\E_s(\epsilon^2)\ra\infty$.
\end{definition}

Policy maker's objective is to minimize the objective function expressed in Equation (\ref{13})  subject to the dynamic system represented by the equations (\ref{0}), (\ref{5}) and (\ref{7}). Following \cite{pramanik2020optimization} the quantum Lagrangian of risk-group $k$ can be expressed as 
\begin{multline}\label{14}
\mathcal{L}_k(s,\rho,\theta_k,\breve\chi_k,\check h,\varpi_k,e_k,z_k,S_k,I_k,R_k,W_k)\\=\E_s\biggr\{\exp\{-\rho s\}\sum_{k=1}^K\theta_kz_k(s)\left[N_k-\mathcal E_k(s)\right]+\breve\chi_k\check h\varpi_k(s)I_k(s)\\+\lambda_1\left[\Delta z_k(s)-[\kappa_0\{1-e_k(s)\}-\kappa_1z_k(s)p(\eta_{k_i},s)]ds-\sigma_0^k[z_k(s)-z_k^*]dB_0^k(s)\right]\\+\lambda_2\left[\Delta S_k(s)-\mu_1(s,e_k,z_k,S_k,I_k,R_k)ds-\sigma_5^k(S_k)dB_2^k\right]\\+\lambda_3\left[\Delta I_k(s)-\mu_2(s,e_k,z_k,S_k,I_k,R_k)ds-\sigma_6^k(I_k)dB_2^k\right]\\+\lambda_4\left[\Delta R_k(s)-\mu_3(s,e_k,z_k,S_k,I_k,R_k)ds-\sigma_7^k(R_k)dB_2^k\right]\\+\lambda_5\left[\Delta W_k(s)-\mu_4(s,e_k,z_k,S_k,I_k,R_k)ds+\sigma_{10}^k(s,e_k,\omega_k,\omega_l)dB_2^k\right]\biggr\},
\end{multline}
where $\lambda_i>0$ for all $i=\{1,2,3,4\}$ are time independent quantum Lagrangian multipliers and $\Delta$'s represent small change of state variables in time interval $(s,s+\varepsilon)$ for all $\varepsilon>0$ and $\varepsilon \searrow 0$.  As $\lambda$'s do not depend on time,  they are considered as penalization constants. At time $s$ risk-group $k$ can predict based on all information available regarding state variables at that time, throughout interval $[s,s+\varepsilon]$ it has the same conditional expectation which ultimately gets rid of the integration.

\medskip

\section{Main results}
In this section I am going to determine an optimal lock intensity for risk-group $k$. By using Feynman-type path integral approach I find a Euclidean action function, define a transition wave function and finally, I derive a Fokker-Plank-type (i.e. Wick-rotated Schr\"odinger-type) equation of the system. 

\begin{prop}\label{p2}
Suppose, the domain of the quantum Lagrangian $\mathcal L_k$ has a non-empty, convex and compact denoted as $\widetilde\Xi$ such that $\widetilde\Xi\subset \mathbb R^{6K}\times G$. As $\mathcal L_k: \widetilde\Xi\ra\widetilde\Xi$ is continuous, then for any given positive constants $\rho,\theta_k,\breve\xi_k,\breve h$ and $\varpi_k$, there exists a vector of state and control variables $\bar Z_k^*=[e_k^*,z_k^*,S_k^*,I_k^*,R_k^*,W_k^*]^T$ in continouous time $s\in[0,t]$ such that $\mathcal L_k$ has a fixed-point in Brouwer sense, where $T$ denotes the transposition of a matrix.
\end{prop}
\begin{proof}
	See in the Appendix.
\end{proof}
Proposition \ref{p2} guarantees that the pandemic control problem at least one fixed point, which leads to the next Theorem \ref{t0}. Theorem \ref{t0} is the main result of this paper. It uses a Euclidean path integral approach based on a Feynman-type action function to get an optimal ``lock-down" intensity.

\begin{theorem}\label{t0}
Suppose, for all $k\in\{1,2,...,K\}$ a social planner's objective is to minimize $\mathbf H_\theta^k$ subject to the stochastic dynamic system explained in the Equations (\ref{0}), (\ref{4}) and (\ref{6}) such that the Assumptions (\ref{as0})- (\ref{as2}) and Propositions \ref{p0}-\ref{p2} hold. For a $C^2$-function $\tilde f_k(s,e_k,z_k,S_k,I_k,R_k,W_k)$ and for all $s\in[0,t]$ there exists a function  $g_k(z_k,S_k,I_k,R_k,W_k)\in C^2([0,t]\times\mathbb{R}^{5K})$ such that  $\widetilde Y_k=g_k[z_k,S_k,I_k,R_k,W_k]$, with  an It\^o process $\widetilde Y_k$, and for a non-singular matrix 
\[
\bm\Theta_k=\mbox{$\frac{1}{2}$}\begin{bmatrix}
\frac{\partial^2\tilde f_k}{\partial z_k^2} & \frac{\partial^2\tilde f_k}{\partial z_k\partial S_k} & \frac{\partial^2\tilde f_k}{\partial z_k\partial I_k}&\frac{\partial^2\tilde f_k}{\partial z_k\partial R_k}& \frac{\partial^2\tilde f_k}{\partial z_k\partial W_k}\\
\frac{\partial^2\tilde f_k}{\partial S_k\partial z_k}&\frac{\partial^2\tilde f_k}{\partial S_k^2}&\frac{\partial^2\tilde f_k}{\partial S_k\partial I_k}&\frac{\partial^2\tilde f_k}{\partial S_k\partial R_k}&\frac{\partial^2\tilde f_k}{\partial S_k\partial W_k}\\
\frac{\partial^2\tilde f_k}{\partial I_k\partial z_k}& \frac{\partial^2\tilde f_k}{\partial I_k\partial S_k}&\frac{\partial^2\tilde f_k}{\partial I_k^2}&\frac{\partial^2\tilde f_k}{\partial I_k\partial R_k}&\frac{\partial^2\tilde f_k}{\partial I_k\partial W_k}\\
\frac{\partial^2\tilde f_k}{\partial R_k\partial z_k}&\frac{\partial^2\tilde f_k}{\partial R_k\partial S_k}&\frac{\partial^2\tilde f_k}{\partial R_k\partial I_k}&\frac{\partial^2\tilde f_k}{\partial R_k^2}&\frac{\partial^2\tilde f_k}{\partial R_k\partial W_k}\\
\frac{\partial^2\tilde f_k}{\partial W_k\partial z_k}&\frac{\partial^2\tilde f_k}{\partial W_k\partial S_k}&\frac{\partial^2\tilde f_k}{\partial W_k\partial I_k}& \frac{\partial^2\tilde f_k}{\partial W_k\partial R_k}&\frac{\partial^2\tilde f_k}{\partial W_k^2},
\end{bmatrix}
\]
optimal ``lock-down" intensity $e_k^*$ is the solution of the Equation
\begin{equation}\label{22}
-\frac{\partial }{\partial e_k}\tilde f_k(s,e_k,z_k,S_k,I_k,R_k,W_k)\Psi_s^{k\tau}(z_k,S_k,I_k,R_k,W_k)=0,
\end{equation}
where $\Psi_s^{k\tau}$ is some transition wave function in $\{\mathbb R^{5K}\times G\}$.
\end{theorem}

\begin{proof}
	From quantum Lagrangian function expressed in the Equation (\ref{14}), the Euclidean action function for risk-group $k$ in $[0,t]$ is given by
	\begin{multline*}
	\mathcal A_{0,t}^k(z_k,S_k,I_k,R_k,W_k)=\int_0^t\E_s\biggr\{\exp\{-\rho s\}\sum_{k=1}^K\theta_kz_k(s)\left[N_k-\mathcal E_k(s)\right]+\breve\chi_k\check h\varpi_k(s)I_k(s)ds\\+\lambda_1\left[\Delta z_k(s)-[\kappa_0\{1-e_k(s)\}-\kappa_1z_k(s)p(\eta_{k_i},s)]ds-\sigma_0^k[z_k(s)-z_k^*]dB_0^k(s)\right]\\+\lambda_2\left[\Delta S_k(s)-\mu_1(s,e_k,z_k,S_k,I_k,R_k)ds-\sigma_5^k(S_k)dB_2^k\right]\\+\lambda_3\left[\Delta I_k(s)-\mu_2(s,e_k,z_k,S_k,I_k,R_k)ds-\sigma_6^k(I_k)dB_2^k\right]\\+\lambda_4\left[\Delta R_k(s)-\mu_3(s,e_k,z_k,S_k,I_k,R_k)ds-\sigma_7^k(R_k)dB_2^k\right]\\+\lambda_5\left[\Delta W_k(s)-\mu_4(s,e_k,z_k,S_k,I_k,R_k)ds+\sigma_{10}^k(s,e_k,\omega_k,\omega_l)dB_2^k\right]\biggr\},
	\end{multline*}
	where $\lambda_i>0$ for all $i=\{1,2,3,4\}$ are time independent quantum Lagrangian multiplier. As at the beginning of the small time interval $[s,s+\epsilon]$, agent $k$ does not have any future information, they make expectations based on their  all state variables $\{z_k,S_k,I_k,R_k,W_k\}$. For a penalization constant $L_\epsilon>0$ and for time interval $[s,s+\epsilon]$ such that $\epsilon\downarrow 0$ define a transition function from $s$ to $s+\epsilon$ as
	\begin{multline}\label{15}
	\Psi_{s,s+\varepsilon}^k(z_k,S_k,I_k,R_k,W_k)=\frac{1}{L_\varepsilon} \int_{\mathbb{R}^{5K}}\exp[-\varepsilon \mathcal{A}_{s,s+\varepsilon}(z_k,S_k,I_k,R_k,W_k)]\Psi_s^k(z_k,S_k,I_k,R_k,W_k)\\\times dz_k\times dS_k\times dI_k\times dR_k\times dW_k,
	\end{multline}
	where $\Psi_s^k(z_k,S_k,I_k,R_k,W_k)$ is the value of the transition function at time $s$ with the initial condition 
	\[
	\Psi_0^k(z_k,S_k,I_k,R_k,W_k)=\Psi_0^k
	\]
	 and the action function of risk-group $k$ is, 
	\begin{multline*}
	\mathcal A_{s,s+\varepsilon}(z_k,S_k,I_k,R_k,W_k)=\int_s^{s+\varepsilon}\E_\nu\biggr\{\left[\exp\{-\rho \nu\}\sum_{k=1}^K\theta_kz_k(\nu)\left[N_k-\mathcal E_k(\nu)\right]+\breve\chi_k\check h\varpi_k(\nu)I_k(\nu)\right]d\nu\\+g_k\left[\nu+\Delta \nu, S_k(\nu)+\Delta S_k(\nu),I_K+ \Delta I_k(\nu),R_k(\nu)+\Delta R_k(\nu), W_k(\nu)+\Delta W_k(\nu)\right]\biggr\},
	\end{multline*}
	where $g_k(z_k,S_k,I_k,R_k,W_k)\in C^2([0,t]\times\mathbb{R}^{5K})$ such that Assumptions \ref{as0}- \ref{as2} hold and $\widetilde Y_k(\nu)=g_k[z_k,S_k,I_k,R_k,W_k]$, where $\widetilde Y_k$ is an It\^o process \citep{oksendal2003} and,
	\begin{multline*}
	g_k(z_k,S_k,I_k,R_k,W_k)\\=\lambda_1\left[\Delta z_k(s)-[\kappa_0\{1-e_k(s)\}-\kappa_1z_k(s)p(\eta_{k_i},s)]ds-\sigma_0^k[z_k(s)-z_k^*]dB_0^k(s)\right]\\+\lambda_2\left[\Delta S_k(s)-\mu_1(s,e_k,z_k,S_k,I_k,R_k)ds-\sigma_5^k(S_k)dB_2^k\right]\\+\lambda_3\left[\Delta I_k(s)-\mu_2(s,e_k,z_k,S_k,I_k,R_k)ds-\sigma_6^k(I_k)dB_2^k\right]\\+\lambda_4\left[\Delta R_k(s)-\mu_3(s,e_k,z_k,S_k,I_k,R_k)ds-\sigma_7^k(R_k)dB_2^k\right]\\+\lambda_5\left[\Delta W_k(s)-\mu_4(s,e_k,z_k,S_k,I_k,R_k)ds+\sigma_{10}^k(s,e_k,\omega_k,\omega_l)dB_2^k\right]+o(1),
	\end{multline*}
	where $\Delta z_k=z_k(s+\varepsilon)-z_k(s)$, $\Delta S_k=S_k(s+\varepsilon)-S_k(s)$, $\Delta I_k=I_k(s+\varepsilon)-I_k(s)$, $\Delta R_k=R_k(s+\varepsilon)-R_k(s)$ and $\Delta W_k=W_k(s+\varepsilon)-W_k(s)$. In Equation (\ref{15}) $L_\varepsilon$ is a positive penalization constant such that the value of $\Psi_{s,s+\varepsilon}^k(.)$ becomes $1$. One can think this transition function $\Psi_{s,s+\varepsilon}^k(.)$ as some transition probability function on Euclidean space. I divide the time interval $[0,t]$ into $n$ small equal length time intervals $[s,s+\varepsilon]$ such that $\tau=s+\varepsilon$. After using Fubini's Theorem, the Euclidean action function for time interval $[s,\tau]$ becomes,
	\begin{multline*}
	\mathcal A_{s,\tau}(z_k,S_k,I_k,R_k,W_k)=\E_s\biggr\{\int_s^{\tau}\left[\exp\{-\rho \nu\}\sum_{k=1}^K\theta_kz_k(\nu)\left[N_k-\mathcal E_k(\nu)\right]+\breve\chi_k\check h\varpi_k(\nu)I_k(\nu)\right]d\nu\\+g_k\left[\nu+\Delta \nu, S_k(\nu)+\Delta S_k(\nu),I_K+ \Delta I_k(\nu),R_k(\nu)+\Delta R_k(\nu), W_k(\nu)+\Delta W_k(\nu)\right]\biggr\}.
	\end{multline*}
	
 After using the fact that $[\Delta z_k(s)]^2=[\Delta S_k(s)]^2=[\Delta I_k(s)]^2=[\Delta R_k(s)]^2=[\Delta W_k(s)]^2=\varepsilon$, and $\E_s[\Delta B_0^k]=\E_s[\Delta B_2^k]=\E_s[\Delta B_3^k]=\E_s[\Delta B_4^k]=\E_s[\Delta B_7^k]$ for all $\varepsilon\downarrow 0$, (with initial conditions $z_k(0),S_k(0),I_k(0),R_k(0),W_k(0)$) It\^o's formula and \cite{baaquie1997} imply,
 
 \begin{multline*}
 \mathcal A_{s,\tau}(z_k,S_k,I_k,R_k,W_k)=\exp\{-\rho s\}\sum_{k=1}^K\theta_kz_k(s)\left[N_k-\mathcal E_k(s)\right]+\breve\chi_k\check h\varpi_k(s)I_k(s)\\+g_k+\frac{\partial}{\partial s}g_k+\frac{\partial}{\partial z_k}g_k\times[\kappa_0\{1-e_k(s)\}-\kappa_1z_k(s)p(\eta_{k_i},s)]\\+\frac{\partial}{\partial S_k}g_k\mu_1(s,e_k,z_k,S_k,I_k,R_k)+\frac{\partial}{\partial I_k}g_k\mu_2(s,e_k,z_k,S_k,I_k,R_k)\\+\frac{\partial}{\partial R_k}g_k \mu_3(s,e_k,z_k,S_k,I_k,R_k)+\frac{\partial}{\partial W_k}g_k\mu_4(s,e_k,z_k,S_k,I_k,R_k)\\+\frac{1}{2}\left\{[\sigma_0^k(z_k(s)-z_k^*)]^2\frac{\partial^2}{\partial z_k^2}g_k+[\sigma_5^k(S_k)]^2\frac{\partial^2}{\partial S_k^2}g_k\right.\\\left.+[\sigma_6^k(I_k)]^2\frac{\partial^2}{\partial I_k^2}g_k+[\sigma_7^k(R_k)]^2\frac{\partial^2}{\partial R_k^2}g_k\right.\\\left.+[\sigma_{10}^k(s,e_k,\omega_k,\omega_l)]^2\frac{\partial^2}{\partial W_k^2}g_k+2\left[\sigma_5^k(S_k)[\sigma_0^k(z_k(s)-z_k^*)]\right.\right.\\\left.\left.\times\frac{\partial^2}{\partial z_k\partial S_k}g_k+\sigma_6^k(I_k)[\sigma_0^k(z_k(s)-z_k^*)]\frac{\partial^2}{\partial z_k\partial I_k}g_k\right.\right.\\\left.\left.+\sigma_7^k(R_k)[\sigma_0^k(z_k(s)-z_k^*)]\frac{\partial^2}{\partial z_k\partial R_k}g_k\right.\right.\\\left.\left.+[\sigma_0^k(z_k(s)-z_k^*)]\sigma_{10}^k(s,e_k,\omega_k,\omega_l)\frac{\partial^2}{\partial z_k\partial W_k}g_k\right.\right.\\\left.\left.+\sigma_5^k(S_k)\sigma_6^k(I_k)\frac{\partial^2}{\partial S_k\partial I_k}g_k+\sigma_5^k(S_k)\sigma_7^k(R_k)\frac{\partial^2}{\partial S_k\partial R_k}g_k\right.\right.\\\left.\left.+\sigma_5^k(S_k)\sigma_{10}^k(s,e_k,\omega_k,\omega_l)\frac{\partial^2}{\partial S_k\partial W_k}g_k+\sigma_6^k(I_k)\sigma_7^k(R_k)\right.\right.\\\left.\left.\times\frac{\partial^2}{\partial I_k\partial R_k}g_k+\sigma_6^k(I_k)\sigma_{10}^k(s,e_k,\omega_k,\omega_l)\frac{\partial^2}{\partial I_k\partial W_k}g_k\right.\right.\\\left.\left.+\sigma_7^k(R_k)\sigma_{10}^k(s,e_k,\omega_k,\omega_l)\frac{\partial^2}{\partial R_k\partial W_k}g_k\right]\right\}+o(1),
 \end{multline*}
 where $g_k=g_k(z_k,S_k,I_k,R_k,W_k)$.
 
 Result in Equation(\ref{15}) implies,
 \begin{multline}\label{16}
 \Psi_{s,s+\varepsilon}^k(z_k,S_k,I_k,R_k,W_k)\\=\frac{1}{L_\varepsilon} \int_{\mathbb{R}^{5K}}\exp\biggm[-\varepsilon \biggr[\exp\{-\rho s\}\sum_{k=1}^K\theta_kz_k(s)\left[N_k-\mathcal E_k(s)\right]+\breve\chi_k\check h\varpi_k(s)I_k(s)\\+g_k+\frac{\partial}{\partial s}g_k+\frac{\partial}{\partial z_k}g_k\times[\kappa_0\{1-e_k(s)\}-\kappa_1z_k(s)p(\eta_{k_i},s)]\\+\frac{\partial}{\partial S_k}g_k\mu_1(s,e_k,z_k,S_k,I_k,R_k)+\frac{\partial}{\partial I_k}g_k\mu_2(s,e_k,z_k,S_k,I_k,R_k)\\+\frac{\partial}{\partial R_k}g_k \mu_3(s,e_k,z_k,S_k,I_k,R_k)+\frac{\partial}{\partial W_k}g_k\mu_4(s,e_k,z_k,S_k,I_k,R_k)\\+\frac{1}{2}\left\{[\sigma_0^k(z_k(s)-z_k^*)]^2\frac{\partial^2}{\partial z_k^2}g_k+[\sigma_5^k(S_k)]^2\frac{\partial^2}{\partial S_k^2}g_k\right.\\\left.+[\sigma_6^k(I_k)]^2\frac{\partial^2}{\partial I_k^2}g_k+[\sigma_7^k(R_k)]^2\frac{\partial^2}{\partial R_k^2}g_k\right.\\\left.+[\sigma_{10}^k(s,e_k,\omega_k,\omega_l)]^2\frac{\partial^2}{\partial W_k^2}g_k+2\left[\sigma_5^k(S_k)[\sigma_0^k(z_k(s)-z_k^*)]\right.\right.\\\left.\left.\times\frac{\partial^2}{\partial z_k\partial S_k}g_k+\sigma_6^k(I_k)[\sigma_0^k(z_k(s)-z_k^*)]\frac{\partial^2}{\partial z_k\partial I_k}g_k\right.\right.\\\left.\left.+\sigma_7^k(R_k)[\sigma_0^k(z_k(s)-z_k^*)]\frac{\partial^2}{\partial z_k\partial R_k}g_k\right.\right.\\\left.\left.+[\sigma_0^k(z_k(s)-z_k^*)]\sigma_{10}^k(s,e_k,\omega_k,\omega_l)\frac{\partial^2}{\partial z_k\partial W_k}g_k\right.\right.\\\left.\left.+\sigma_5^k(S_k)\sigma_6^k(I_k)\frac{\partial^2}{\partial S_k\partial I_k}g_k+\sigma_5^k(S_k)\sigma_7^k(R_k)\frac{\partial^2}{\partial S_k\partial R_k}g_k\right.\right.\\\left.\left.+\sigma_5^k(S_k)\sigma_{10}^k(s,e_k,\omega_k,\omega_l)\frac{\partial^2}{\partial S_k\partial W_k}g_k+\sigma_6^k(I_k)\sigma_7^k(R_k)\right.\right.\\\left.\left.\times\frac{\partial^2}{\partial I_k\partial R_k}g_k+\sigma_6^k(I_k)\sigma_{10}^k(s,e_k,\omega_k,\omega_l)\frac{\partial^2}{\partial I_k\partial W_k}g_k\right.\right.\\\left.\left.+\sigma_7^k(R_k)\sigma_{10}^k(s,e_k,\omega_k,\omega_l)\frac{\partial^2}{\partial R_k\partial W_k}g_k\right]\right\}\biggr]\biggm]\\\times\Psi_s^k(z_k,S_k,I_k,R_k,W_k)\times dz_k\times dS_k\times dI_k\times dR_k\times dW_k+o(\varepsilon ^{1/2}).
 \end{multline}
 For $\varepsilon \downarrow 0$ define a new transition probability $\Psi_s^{k\tau}$ centered around time $\tau$. A Taylor series expansion (up to second order) of the left hand side of Equation (\ref{16}) yields,
 \begin{multline*}
 \Psi_s^{k\tau}(z_k,S_k,I_k,R_k,W_k)+\varepsilon\frac{\partial \Psi_s^{k\tau}(z_k,S_k,I_k,R_k,W_k)}{\partial s}+o(\varepsilon)\\=\frac{1}{L_\varepsilon} \int_{\mathbb{R}^{5K}}\exp\biggm[-\varepsilon \biggr[\exp\{-\rho s\}\sum_{k=1}^K\theta_kz_k(s)\left[N_k-\mathcal E_k(s)\right]+\breve\chi_k\check h\varpi_k(s)I_k(s)\\+g_k+\frac{\partial}{\partial s}g_k+\frac{\partial}{\partial z_k}g_k\times[\kappa_0\{1-e_k(s)\}-\kappa_1z_k(s)p(\eta_{k_i},s)]\\+\frac{\partial}{\partial S_k}g_k\mu_1(s,e_k,z_k,S_k,I_k,R_k)+\frac{\partial}{\partial I_k}g_k\mu_2(s,e_k,z_k,S_k,I_k,R_k)\\+\frac{\partial}{\partial R_k}g_k \mu_3(s,e_k,z_k,S_k,I_k,R_k)+\frac{\partial}{\partial W_k}g_k\mu_4(s,e_k,z_k,S_k,I_k,R_k)\\+\frac{1}{2}\left\{[\sigma_0^k(z_k(s)-z_k^*)]^2\frac{\partial^2}{\partial z_k^2}g_k+[\sigma_5^k(S_k)]^2\frac{\partial^2}{\partial S_k^2}g_k\right.\\\left.+[\sigma_6^k(I_k)]^2\frac{\partial^2}{\partial I_k^2}g_k+[\sigma_7^k(R_k)]^2\frac{\partial^2}{\partial R_k^2}g_k\right.\\\left.+[\sigma_{10}^k(s,e_k,\omega_k,\omega_l)]^2\frac{\partial^2}{\partial W_k^2}g_k+2\left[\sigma_5^k(S_k)[\sigma_0^k(z_k(s)-z_k^*)]\right.\right.\\\left.\left.\times\frac{\partial^2}{\partial z_k\partial S_k}g_k+\sigma_6^k(I_k)[\sigma_0^k(z_k(s)-z_k^*)]\frac{\partial^2}{\partial z_k\partial I_k}g_k\right.\right.\\\left.\left.+\sigma_7^k(R_k)[\sigma_0^k(z_k(s)-z_k^*)]\frac{\partial^2}{\partial z_k\partial R_k}g_k\right.\right.\\\left.\left.+[\sigma_0^k(z_k(s)-z_k^*)]\sigma_{10}^k(s,e_k,\omega_k,\omega_l)\frac{\partial^2}{\partial z_k\partial W_k}g_k\right.\right.\\\left.\left.+\sigma_5^k(S_k)\sigma_6^k(I_k)\frac{\partial^2}{\partial S_k\partial I_k}g_k+\sigma_5^k(S_k)\sigma_7^k(R_k)\frac{\partial^2}{\partial S_k\partial R_k}g_k\right.\right.\\\left.\left.+\sigma_5^k(S_k)\sigma_{10}^k(s,e_k,\omega_k,\omega_l)\frac{\partial^2}{\partial S_k\partial W_k}g_k+\sigma_6^k(I_k)\sigma_7^k(R_k)\right.\right.\\\left.\left.\times\frac{\partial^2}{\partial I_k\partial R_k}g_k+\sigma_6^k(I_k)\sigma_{10}^k(s,e_k,\omega_k,\omega_l)\frac{\partial^2}{\partial I_k\partial W_k}g_k\right.\right.\\\left.\left.+\sigma_7^k(R_k)\sigma_{10}^k(s,e_k,\omega_k,\omega_l)\frac{\partial^2}{\partial R_k\partial W_k}g_k\right]\right\}\biggr]\biggm]\\\times\Psi_s^k(z_k,S_k,I_k,R_k,W_k)\times dz_k\times dS_k\times dI_k\times dR_k\times dW_k+o(\varepsilon ^{1/2}),
  \end{multline*}
as $\varepsilon\downarrow 0$. For fixed $s$ and $\tau$ let $z_k(s)=z_k(\tau)+\varsigma_1$, $S_k(s)=S_k(\tau) +\varsigma_2$, $I_k(s)=I_k(\tau)+\varsigma_3$, $R_k(s)=R_k(\tau)+\varsigma_4$ and $W_k(s)=W_k(\tau)+\varsigma_5$. For some finite positive numbers $c_i$ with $i=1,...,5$ assume $|\varsigma_1|\leq\frac{c_1\varepsilon}{z_k(s)}$, $|\varsigma_2|\leq\frac{c_2\varepsilon}{S_k(s)}$, $|\varsigma_3|\leq\frac{c_3\varepsilon}{I_k(s)}$, $|\varsigma_4|\leq\frac{c_4\varepsilon}{R_k(s)}$ and, $|\varsigma_5|\leq\frac{c_5\varepsilon}{W_k(s)}$. Therefore, we get upper bounds of each state variable in this pandemic control model as $z_k(s)\leq c_1\varepsilon/ (\varsigma_1)^2 $, $S_k(s)\leq c_2\varepsilon/ (\varsigma_2)^2 $, $I_k(s)\leq c_3\varepsilon/ (\varsigma_3)^2 $, $R_k(s)\leq c_4\varepsilon/ (\varsigma_4)^2 $ and $W_k(s)\leq c_5\varepsilon/ (\varsigma_5)^2 $. Furthermore, by Fr\"ohlich's Reconstruction Theorem \citep{simon1979,pramanik2020optimization,pramanik2021optparam} and Assumptions \ref{as0}-\ref{as2} imply

\begin{multline}\label{17}
\Psi_s^{k\tau}(z_k,S_k,I_k,R_k,W_k)+\varepsilon\frac{\partial \Psi_s^{k\tau}(z_k,S_k,I_k,R_k,W_k)}{\partial s}+o(\varepsilon)\\=\frac{1}{L_\varepsilon} \int_{\mathbb{R}^{5K}}\left[\Psi_s^{k\tau}(z_k,S_k,I_k,R_k,W_k)+\varsigma_1\frac{\partial \Psi_s^{k\tau}(z_k,S_k,I_k,R_k,W_k)}{\partial z_k}\right.\\\left.+\varsigma_2\frac{\partial \Psi_s^{k\tau}(z_k,S_k,I_k,R_k,W_k)}{\partial S_k}+\varsigma_3\frac{\partial \Psi_s^{k\tau}(z_k,S_k,I_k,R_k,W_k)}{\partial I_k}\right.\\\left.+\varsigma_4\frac{\partial \Psi_s^{k\tau}(z_k,S_k,I_k,R_k,W_k)}{\partial R_k}+\varsigma_5\frac{\partial \Psi_s^{k\tau}(z_k,S_k,I_k,R_k,W_k)}{\partial W_k}+o(\varepsilon)\right]\\\times \exp\biggm[-\varepsilon \biggr[\exp\{-\rho s\}\sum_{k=1}^K\theta_kz_k(s)\left[N_k-\mathcal E_k(s)\right]+\breve\chi_k\check h\varpi_k(s)I_k(s)\\+g_k+\frac{\partial}{\partial s}g_k+\frac{\partial}{\partial z_k}g_k\times[\kappa_0\{1-e_k(s)\}-\kappa_1z_k(s)p(\eta_{k_i},s)]\\+\frac{\partial}{\partial S_k}g_k\mu_1(s,e_k,z_k,S_k,I_k,R_k)+\frac{\partial}{\partial I_k}g_k\mu_2(s,e_k,z_k,S_k,I_k,R_k)\\+\frac{\partial}{\partial R_k}g_k \mu_3(s,e_k,z_k,S_k,I_k,R_k)+\frac{\partial}{\partial W_k}g_k\mu_4(s,e_k,z_k,S_k,I_k,R_k)\\+\frac{1}{2}\left\{[\sigma_0^k(z_k(s)-z_k^*)]^2\frac{\partial^2}{\partial z_k^2}g_k+[\sigma_5^k(S_k)]^2\frac{\partial^2}{\partial S_k^2}g_k\right.\\\left.+[\sigma_6^k(I_k)]^2\frac{\partial^2}{\partial I_k^2}g_k+[\sigma_7^k(R_k)]^2\frac{\partial^2}{\partial R_k^2}g_k\right.\\\left.+[\sigma_{10}^k(s,e_k,\omega_k,\omega_l)]^2\frac{\partial^2}{\partial W_k^2}g_k+2\left[\sigma_5^k(S_k)[\sigma_0^k(z_k(s)-z_k^*)]\right.\right.\\\left.\left.\times\frac{\partial^2}{\partial z_k\partial S_k}g_k+\sigma_6^k(I_k)[\sigma_0^k(z_k(s)-z_k^*)]\frac{\partial^2}{\partial z_k\partial I_k}g_k\right.\right.\\\left.\left.+\sigma_7^k(R_k)[\sigma_0^k(z_k(s)-z_k^*)]\frac{\partial^2}{\partial z_k\partial R_k}g_k\right.\right.\\\left.\left.+[\sigma_0^k(z_k(s)-z_k^*)]\sigma_{10}^k(s,e_k,\omega_k,\omega_l)\frac{\partial^2}{\partial z_k\partial W_k}g_k\right.\right.\\\left.\left.+\sigma_5^k(S_k)\sigma_6^k(I_k)\frac{\partial^2}{\partial S_k\partial I_k}g_k+\sigma_5^k(S_k)\sigma_7^k(R_k)\frac{\partial^2}{\partial S_k\partial R_k}g_k\right.\right.\\\left.\left.+\sigma_5^k(S_k)\sigma_{10}^k(s,e_k,\omega_k,\omega_l)\frac{\partial^2}{\partial S_k\partial W_k}g_k+\sigma_6^k(I_k)\sigma_7^k(R_k)\right.\right.\\\left.\left.\times\frac{\partial^2}{\partial I_k\partial R_k}g_k+\sigma_6^k(I_k)\sigma_{10}^k(s,e_k,\omega_k,\omega_l)\frac{\partial^2}{\partial I_k\partial W_k}g_k\right.\right.\\\left.\left.+\sigma_7^k(R_k)\sigma_{10}^k(s,e_k,\omega_k,\omega_l)\frac{\partial^2}{\partial R_k\partial W_k}g_k\right]\right\}\biggr]\biggm]\\\times\Psi_s^k(z_k,S_k,I_k,R_k,W_k)\times dz_k\times dS_k\times dI_k\times dR_k\times dW_k+o(\varepsilon ^{1/2}),
\end{multline}
as $\varepsilon\downarrow 0$. For risk-group $k\in\{1,2,...,K\}$ define a function 
\begin{multline*}
\tilde f_k(s,e_k,z_k,S_k,I_k,R_k,W_k)=\exp\{-\rho s\}\sum_{k=1}^K\theta_kz_k(s)\left[N_k-\mathcal E_k(s)\right]+\breve\chi_k\check h\varpi_k(s)I_k(s)\\+g_k+\frac{\partial}{\partial s}g_k+\frac{\partial}{\partial z_k}g_k\times[\kappa_0\{1-e_k(s)\}-\kappa_1z_k(s)p(\eta_{k_i},s)]\\+\frac{\partial}{\partial S_k}g_k\mu_1(s,e_k,z_k,S_k,I_k,R_k)+\frac{\partial}{\partial I_k}g_k\mu_2(s,e_k,z_k,S_k,I_k,R_k)\\+\frac{\partial}{\partial R_k}g_k \mu_3(s,e_k,z_k,S_k,I_k,R_k)+\frac{\partial}{\partial W_k}g_k\mu_4(s,e_k,z_k,S_k,I_k,R_k)\\+\frac{1}{2}\left\{[\sigma_0^k(z_k(s)-z_k^*)]^2\frac{\partial^2}{\partial z_k^2}g_k+[\sigma_5^k(S_k)]^2\frac{\partial^2}{\partial S_k^2}g_k\right.\\\left.+[\sigma_6^k(I_k)]^2\frac{\partial^2}{\partial I_k^2}g_k+[\sigma_7^k(R_k)]^2\frac{\partial^2}{\partial R_k^2}g_k\right.\\\left.+[\sigma_{10}^k(s,e_k,\omega_k,\omega_l)]^2\frac{\partial^2}{\partial W_k^2}g_k+2\left[\sigma_5^k(S_k)[\sigma_0^k(z_k(s)-z_k^*)]\right.\right.\\\left.\left.\times\frac{\partial^2}{\partial z_k\partial S_k}g_k+\sigma_6^k(I_k)[\sigma_0^k(z_k(s)-z_k^*)]\frac{\partial^2}{\partial z_k\partial I_k}g_k\right.\right.\\\left.\left.+\sigma_7^k(R_k)[\sigma_0^k(z_k(s)-z_k^*)]\frac{\partial^2}{\partial z_k\partial R_k}g_k\right.\right.\\\left.\left.+[\sigma_0^k(z_k(s)-z_k^*)]\sigma_{10}^k(s,e_k,\omega_k,\omega_l)\frac{\partial^2}{\partial z_k\partial W_k}g_k\right.\right.\\\left.\left.+\sigma_5^k(S_k)\sigma_6^k(I_k)\frac{\partial^2}{\partial S_k\partial I_k}g_k+\sigma_5^k(S_k)\sigma_7^k(R_k)\frac{\partial^2}{\partial S_k\partial R_k}g_k\right.\right.\\\left.\left.+\sigma_5^k(S_k)\sigma_{10}^k(s,e_k,\omega_k,\omega_l)\frac{\partial^2}{\partial S_k\partial W_k}g_k+\sigma_6^k(I_k)\sigma_7^k(R_k)\right.\right.\\\left.\left.\times\frac{\partial^2}{\partial I_k\partial R_k}g_k+\sigma_6^k(I_k)\sigma_{10}^k(s,e_k,\omega_k,\omega_l)\frac{\partial^2}{\partial I_k\partial W_k}g_k\right.\right.\\\left.\left.+\sigma_7^k(R_k)\sigma_{10}^k(s,e_k,\omega_k,\omega_l)\frac{\partial^2}{\partial R_k\partial W_k}g_k\right]\right\}.
\end{multline*}

Therefore, after using the function $\tilde f(s,e_k,z_k,S_k,I_k,R_k,W_k)$ Equation (\ref{17}) yields,
\begin{multline*}
\Psi_s^{k\tau}(z_k,S_k,I_k,R_k,W_k)+\varepsilon\frac{\partial \Psi_s^{k\tau}(z_k,S_k,I_k,R_k,W_k)}{\partial s}+o(\varepsilon)\\=\frac{1}{L_\varepsilon}\Psi_s^{k\tau}(z_k,S_k,I_k,R_k,W_k)\int_{\mathbb R^{5K}}\exp\left\{-\varepsilon\tilde f_k(s,e_k,\varsigma_1,\varsigma_2,\varsigma_3,\varsigma_4,\varsigma_5)\right\}\prod_{i=1}^5d\varsigma_i\\+\frac{1}{L_\varepsilon}\frac{\partial \Psi_s^{k\tau}(z_k,S_k,I_k,R_k,W_k)}{\partial z_k}\int_{\mathbb R^{5K}}\varsigma_1\exp\left\{-\varepsilon\tilde f_k(s,e_k,\varsigma_1,\varsigma_2,\varsigma_3,\varsigma_4,\varsigma_5)\right\}\prod_{i=1}^5d\varsigma_i\\+\frac{1}{L_\varepsilon}\frac{\partial \Psi_s^{k\tau}(z_k,S_k,I_k,R_k,W_k)}{\partial S_k}\int_{\mathbb R^{5K}}\varsigma_2\exp\left\{-\varepsilon\tilde f_k(s,e_k,\varsigma_1,\varsigma_2,\varsigma_3,\varsigma_4,\varsigma_5)\right\}\prod_{i=1}^5d\varsigma_i\\+\frac{1}{L_\varepsilon}\frac{\partial \Psi_s^{k\tau}(z_k,S_k,I_k,R_k,W_k)}{\partial I_k}\int_{\mathbb R^{5K}}\varsigma_3\exp\left\{-\varepsilon\tilde f_k(s,e_k,\varsigma_1,\varsigma_2,\varsigma_3,\varsigma_4,\varsigma_5)\right\}\prod_{i=1}^5d\varsigma_i\\+\frac{1}{L_\varepsilon}\frac{\partial \Psi_s^{k\tau}(z_k,S_k,I_k,R_k,W_k)}{\partial R_k}\int_{\mathbb R^{5K}}\varsigma_4\exp\left\{-\varepsilon\tilde f_k(s,e_k,\varsigma_4,\varsigma_2,\varsigma_3,\varsigma_4,\varsigma_5)\right\}\prod_{i=1}^5d\varsigma_i\\+\frac{1}{L_\varepsilon}\frac{\partial \Psi_s^{k\tau}(z_k,S_k,I_k,R_k,W_k)}{\partial W_k}\int_{\mathbb R^{5K}}\varsigma_5\exp\left\{-\varepsilon\tilde f_k(s,e_k,\varsigma_1,\varsigma_2,\varsigma_3,\varsigma_4,\varsigma_5)\right\}\prod_{i=1}^5d\varsigma_i+o(\varepsilon^{1/2}).
\end{multline*}
Consider $f_k(s,e_k,\varsigma_1,\varsigma_2,\varsigma_3,\varsigma_4,\varsigma_5)$ is a $C^2$-function, then doing the Taylor series expansion up to second order yields
\begin{multline*}
\tilde f_k(s,e_k(s),\varsigma_1,\varsigma_2,\varsigma_3,\varsigma_4,\varsigma_5)=\tilde f_k(s,e_k(s),\varsigma_1(\tau),\varsigma_2(\tau),\varsigma_3(\tau),\varsigma_4(\tau),\varsigma_5(\tau))\\+[\varsigma_1-z_k(\tau)]\frac{\partial }{\partial z_k}\tilde f_k(s,e_k(s),\varsigma_1(\tau),\varsigma_2(\tau),\varsigma_3(\tau),\varsigma_4(\tau),\varsigma_5(\tau))\\+[\varsigma_2-S_k(\tau)]\frac{\partial }{\partial S_k}\tilde f_k(s,e_k(s),\varsigma_1(\tau),\varsigma_2(\tau),\varsigma_3(\tau),\varsigma_4(\tau),\varsigma_5(\tau))\\+[\varsigma_3-I_k(\tau)]\frac{\partial }{\partial I_k}\tilde f_k(s,e_k(s),\varsigma_1(\tau),\varsigma_2(\tau),\varsigma_3(\tau),\varsigma_4(\tau),\varsigma_5(\tau))\\+[\varsigma_4-R_k(\tau)]\frac{\partial }{\partial R_k}\tilde f_k(s,e_k(s),\varsigma_1(\tau),\varsigma_2(\tau),\varsigma_3(\tau),\varsigma_4(\tau),\varsigma_5(\tau))\\+[\varsigma_5-W_k(\tau)]\frac{\partial }{\partial W_k}\tilde f_k(s,e_k(s),\varsigma_1(\tau),\varsigma_2(\tau),\varsigma_3(\tau),\varsigma_4(\tau),\varsigma_5(\tau))\\+\mbox{$\frac{1}{2}$}(\Xi_1+2\Xi_2)+o(\varepsilon),
\end{multline*}
where
\begin{multline*}
\Xi_1=[\varsigma_1-z_k(\tau)]^2\frac{\partial^2 }{\partial z_k^2}\tilde f_k(s,e_k(s),\varsigma_1(\tau),\varsigma_2(\tau),\varsigma_3(\tau),\varsigma_4(\tau),\varsigma_5(\tau))\\+[\varsigma_2-S_k(\tau)]^2\frac{\partial^2 }{\partial S_k^2}\tilde f_k(s,e_k(s),\varsigma_1(\tau),\varsigma_2(\tau),\varsigma_3(\tau),\varsigma_4(\tau),\varsigma_5(\tau))\\+[\varsigma_3-I_k(\tau)]^2\frac{\partial^2 }{\partial I_k^2}\tilde f_k(s,e_k(s),\varsigma_1(\tau),\varsigma_2(\tau),\varsigma_3(\tau),\varsigma_4(\tau),\varsigma_5(\tau))\\+[\varsigma_4-R_k(\tau)]^2\frac{\partial^2 }{\partial R_k^2}\tilde f_k(s,e_k(s),\varsigma_1(\tau),\varsigma_2(\tau),\varsigma_3(\tau),\varsigma_4(\tau),\varsigma_5(\tau))\\+[\varsigma_5-W_k(\tau)]^2\frac{\partial^2 }{\partial W_k^2}\tilde f_k(s,e_k(s),\varsigma_1(\tau),\varsigma_2(\tau),\varsigma_3(\tau),\varsigma_4(\tau),\varsigma_5(\tau)),
\end{multline*}
and,
\begin{multline*}
\Xi_2=[\varsigma_1-z_k(\tau)][\varsigma_2-S_k(\tau)]\frac{\partial^2 }{\partial z_k\partial S_k}\tilde f_k(s,e_k(s),\varsigma_1(\tau),\varsigma_2(\tau),\varsigma_3(\tau),\varsigma_4(\tau),\varsigma_5(\tau))\\+[\varsigma_1-z_k(\tau)][\varsigma_3-I_k(\tau)]\frac{\partial^2 }{\partial z_k\partial I_k}\tilde f_k(s,e_k(s),\varsigma_1(\tau),\varsigma_2(\tau),\varsigma_3(\tau),\varsigma_4(\tau),\varsigma_5(\tau))\\+[\varsigma_1-z_k(\tau)][\varsigma_4-R_k(\tau)]\frac{\partial^2 }{\partial z_k\partial R_k}\tilde f_k(s,e_k(s),\varsigma_1(\tau),\varsigma_2(\tau),\varsigma_3(\tau),\varsigma_4(\tau),\varsigma_5(\tau))\\+[\varsigma_1-z_k(\tau)][\varsigma_5-W_k(\tau)]\frac{\partial^2 }{\partial z_k\partial W_k}\tilde f_k(s,e_k(s),\varsigma_1(\tau),\varsigma_2(\tau),\varsigma_3(\tau),\varsigma_4(\tau),\varsigma_5(\tau))\\+[\varsigma_2-S_k(\tau)][\varsigma_3-I_k(\tau)]\frac{\partial^2 }{\partial S_k\partial I_k}\tilde f_k(s,e_k(s),\varsigma_1(\tau),\varsigma_2(\tau),\varsigma_3(\tau),\varsigma_4(\tau),\varsigma_5(\tau))\\+[\varsigma_2-S_k(\tau)][\varsigma_4-R_k(\tau)]\frac{\partial^2 }{\partial S_k\partial R_k}\tilde f_k(s,e_k(s),\varsigma_1(\tau),\varsigma_2(\tau),\varsigma_3(\tau),\varsigma_4(\tau),\varsigma_5(\tau))\\+[\varsigma_2-S_k(\tau)][\varsigma_5-W_k(\tau)]\frac{\partial^2 }{\partial S_k\partial W_k}\tilde f_k(s,e_k(s),\varsigma_1(\tau),\varsigma_2(\tau),\varsigma_3(\tau),\varsigma_4(\tau),\varsigma_5(\tau))\\+[\varsigma_3-I_k(\tau)][\varsigma_4-R_k(\tau)]\frac{\partial^2 }{\partial I_k\partial R_k}\tilde f_k(s,e_k(s),\varsigma_1(\tau),\varsigma_2(\tau),\varsigma_3(\tau),\varsigma_4(\tau),\varsigma_5(\tau))\\+[\varsigma_3-I_k(\tau)][\varsigma_5-W_k(\tau)]\frac{\partial^2 }{\partial I_k\partial W_k}\tilde f_k(s,e_k(s),\varsigma_1(\tau),\varsigma_2(\tau),\varsigma_3(\tau),\varsigma_4(\tau),\varsigma_5(\tau))\\+[\varsigma_4-R_k(\tau)][\varsigma_5-W_k(\tau)]\frac{\partial^2 }{\partial R_k\partial W_k}\tilde f_k(s,e_k(s),\varsigma_1(\tau),\varsigma_2(\tau),\varsigma_3(\tau),\varsigma_4(\tau),\varsigma_5(\tau)),
\end{multline*}
as $\varepsilon\downarrow 0$ and $\Delta e_k(s)\downarrow 0$. Define $\breve m_1=\varsigma_1-z_k$, $\breve m_2=\varsigma_2-S_k$, $\breve m_3=\varsigma_3-I_k$, $\breve m_4=\varsigma_4-R_k$ and, $\breve m_5=\varsigma_5-W_k$ such that $d\breve m_i=d\varsigma_i$ for all $i=\{1,...,5\}$. Therefore, after denoting $\tilde f_k(s,e_k(s),\varsigma_1(\tau),\varsigma_2(\tau),\varsigma_3(\tau),\varsigma_4(\tau),\varsigma_5(\tau))=\tilde f_k$ above expression becomes
\begin{multline}\label{18}
\int_{\mathbb R^{5K}}\exp\left\{-\varepsilon \tilde f_k(s,e_k,S_k,I_k,R_k,W_k)\right\}\prod_{i=1}^5d\varsigma_i\\=\int_{\mathbb R^{5K}}\biggr\{-\varepsilon\left[\tilde f_k+\breve m_1\frac{\partial \tilde f_k}{\partial z_k}+\breve m_2\frac{\partial \tilde f_k}{\partial S_k}+\breve m_3\frac{\partial \tilde f_k}{\partial I_k}+\breve m_4\frac{\partial \tilde f_k}{\partial R_k}+\breve m_5\frac{\partial \tilde f_k}{\partial W_k}\right.\\\left.+\mbox{$\frac{1}{2}$}\left(\breve m_1^2\frac{\partial^2\tilde f_k}{\partial z_k^2}+\breve m_2^2\frac{\partial^2\tilde f_k}{\partial S_k^2}+\breve m_3^2\frac{\partial^2\tilde f_k}{\partial I_k^2}+\breve m_4^2\frac{\partial^2\tilde f_k}{\partial R_k^2}+\breve m_5^2\frac{\partial^2\tilde f_k}{\partial W_k^2}\right.\right.\\\left.\left.+2\left[\breve m_1\breve m_2\frac{\partial^2\tilde f_k}{\partial z_k\partial S_k}+\breve m_1\breve m_3\frac{\partial^2\tilde f_k}{\partial z_k\partial I_k}+\breve m_1\breve m_4\frac{\partial^2\tilde f_k}{\partial z_k\partial R_k}+\breve m_1\breve m_5\frac{\partial^2\tilde f_k}{\partial z_k\partial W_k}\right.\right.\right.\\\left.\left.\left.+\breve m_2\breve m_3\frac{\partial^2\tilde f_k}{\partial S_k\partial I_k}+\breve m_2\breve m_4\frac{\partial^2\tilde f_k}{\partial S_k\partial R_k}+\breve m_2\breve m_5\frac{\partial^2\tilde f_k}{\partial S_k\partial W_k}+\breve m_3\breve m_4\frac{\partial^2\tilde f_k}{\partial I_k\partial R_k}\right.\right.\right.\\\left.\left.\left.+\breve m_3\breve m_5\frac{\partial^2\tilde f_k}{\partial I_k\partial W_k}+\breve m_4\breve m_5\frac{\partial^2\tilde f_k}{\partial R_k\partial W_k}\right]\right)\right]\biggr\}\prod_{i=1}^5 d\varsigma_i.
\end{multline}
Let
\[
\bm\Theta_k=\mbox{$\frac{1}{2}$}\begin{bmatrix}
\frac{\partial^2\tilde f_k}{\partial z_k^2} & \frac{\partial^2\tilde f_k}{\partial z_k\partial S_k} & \frac{\partial^2\tilde f_k}{\partial z_k\partial I_k}&\frac{\partial^2\tilde f_k}{\partial z_k\partial R_k}& \frac{\partial^2\tilde f_k}{\partial z_k\partial W_k}\\
\frac{\partial^2\tilde f_k}{\partial S_k\partial z_k}&\frac{\partial^2\tilde f_k}{\partial S_k^2}&\frac{\partial^2\tilde f_k}{\partial S_k\partial I_k}&\frac{\partial^2\tilde f_k}{\partial S_k\partial R_k}&\frac{\partial^2\tilde f_k}{\partial S_k\partial W_k}\\
\frac{\partial^2\tilde f_k}{\partial I_k\partial z_k}& \frac{\partial^2\tilde f_k}{\partial I_k\partial S_k}&\frac{\partial^2\tilde f_k}{\partial I_k^2}&\frac{\partial^2\tilde f_k}{\partial I_k\partial R_k}&\frac{\partial^2\tilde f_k}{\partial I_k\partial W_k}\\
\frac{\partial^2\tilde f_k}{\partial R_k\partial z_k}&\frac{\partial^2\tilde f_k}{\partial R_k\partial S_k}&\frac{\partial^2\tilde f_k}{\partial R_k\partial I_k}&\frac{\partial^2\tilde f_k}{\partial R_k^2}&\frac{\partial^2\tilde f_k}{\partial R_k\partial W_k}\\
\frac{\partial^2\tilde f_k}{\partial W_k\partial z_k}&\frac{\partial^2\tilde f_k}{\partial W_k\partial S_k}&\frac{\partial^2\tilde f_k}{\partial W_k\partial I_k}& \frac{\partial^2\tilde f_k}{\partial W_k\partial R_k}&\frac{\partial^2\tilde f_k}{\partial W_k^2}
\end{bmatrix},
\]
and
\[
{\breve {\mathbf m_k}}=\begin{bmatrix}
\breve m_1\\\breve m_2\\\breve m_3\\\breve m_4\\\breve m_5
\end{bmatrix},
\]
and
\[
-\mathbf J_k=\begin{bmatrix}
\frac{\partial}{\partial z_k}\tilde f_k\\\frac{\partial}{\partial S_k}\tilde f_k\\\frac{\partial}{\partial I_k}\tilde f_k\\\frac{\partial}{\partial R_k}\tilde f_k\\\frac{\partial}{\partial W_k}\tilde f_k
\end{bmatrix},
\]
where the symmetric matrix $\bm\Theta_k$ is assumed to be positive semi-definite. The integrand in Equation (\ref{18}) becomes a shifted Gaussian integral,
\begin{align*}
& \int_{\mathbb{R}^{5K}}\exp \bigg\{-\epsilon \left(\tilde f_k- \mathbf J_k^T {\breve {\mathbf m_k}}+{\breve {\mathbf m_k}}^T \bm\Theta_k {\breve {\mathbf m_k}}\right)\bigg\} d\breve m_k\\&=\exp\left(-\epsilon \tilde f_k\right)\int_{\mathbb{R}^{5K}}\exp\bigg\{(\varepsilon \mathbf J_k^T ){\breve {\mathbf m_k}}-{\breve {\mathbf m_k}}^T(\varepsilon \bm\Theta_k){\breve {\mathbf m_k}}\bigg\} \\&=\frac{\pi}{\sqrt{\varepsilon | \bm\Theta_k|}}\exp\left[\frac{\varepsilon }{4}\mathbf J_k^T\left(\bm\Theta_k\right)^{-1}\mathbf J_k-\varepsilon \tilde f_k\right],
\end{align*}
where $\mathbf J_k^T$ is the transposition of $\mathbf J_k$, ${\breve {\mathbf m_k}}^T$ is the transposition of ${\breve {\mathbf m_k}}$ and $\left(\bm\Theta_k\right)^{-1}$ is the inverse of $\bm\Theta_k$. Hence,
\begin{align}\label{19}
& \frac{1}{L_\varepsilon}\Psi_s^{k\tau}\int_{\mathbb R^{5K}}\exp\{-\varepsilon\tilde f_k\}\prod_{i=1}^5d\varsigma_i\notag\\&=\frac{1}{L_\varepsilon}\Psi_s^{k\tau}\frac{\pi}{\sqrt{\varepsilon | \bm\Theta_k|}}\exp\left[\frac{\varepsilon }{4}\mathbf J_k^T\left(\bm\Theta_k\right)^{-1}\mathbf J_k-\varepsilon \tilde f_k\right],
\end{align}
such that the inverse matrix $\left(\bm\Theta_k\right)^{-1}>0$ exists. Similarly,

\begin{align}\label{20}
&\frac{1}{L_\varepsilon}\frac{\partial\Psi_s^{k\tau}}{\partial z_k}\int_{\mathbb R^{5K}}\varsigma_1\exp\{-\varepsilon\tilde f_k\}\prod_{i=1}^5d\varsigma_i\notag\\&\hspace{2cm}=\frac{1}{L_\varepsilon}\frac{\partial\Psi_s^{k\tau}}{\partial z_k}\frac{\pi}{\sqrt{\varepsilon | \bm\Theta_k|}}\left[\frac{1}{2}\left(\bm\Theta_k\right)^{-1}+z_k\right]\exp\left[\frac{\varepsilon }{4}\mathbf J_k^T\left(\bm\Theta_k\right)^{-1}\mathbf J_k-\varepsilon \tilde f_k\right],\notag\\
&\frac{1}{L_\varepsilon}\frac{\partial\Psi_s^{k\tau}}{\partial S_k}\int_{\mathbb R^{5K}}\varsigma_2\exp\{-\varepsilon\tilde f_k\}\prod_{i=1}^5d\varsigma_i\notag\\&\hspace{2cm}=\frac{1}{L_\varepsilon}\frac{\partial\Psi_s^{k\tau}}{\partial S_k}\frac{\pi}{\sqrt{\varepsilon | \bm\Theta_k|}}\left[\frac{1}{2}\left(\bm\Theta_k\right)^{-1}+S_k\right]\exp\left[\frac{\varepsilon }{4}\mathbf J_k^T\left(\bm\Theta_k\right)^{-1}\mathbf J_k-\varepsilon \tilde f_k\right],\notag\\
&\frac{1}{L_\varepsilon}\frac{\partial\Psi_s^{k\tau}}{\partial I_k}\int_{\mathbb R^{5K}}\varsigma_3\exp\{-\varepsilon\tilde f_k\}\prod_{i=1}^5d\varsigma_i\notag\\&\hspace{2cm}=\frac{1}{L_\varepsilon}\frac{\partial\Psi_s^{k\tau}}{\partial I_k}\frac{\pi}{\sqrt{\varepsilon | \bm\Theta_k|}}\left[\frac{1}{2}\left(\bm\Theta_k\right)^{-1}+I_k\right]\exp\left[\frac{\varepsilon }{4}\mathbf J_k^T\left(\bm\Theta_k\right)^{-1}\mathbf J_k-\varepsilon \tilde f_k\right],\notag\\
&\frac{1}{L_\varepsilon}\frac{\partial\Psi_s^{k\tau}}{\partial R_k}\int_{\mathbb R^{5K}}\varsigma_4\exp\{-\varepsilon\tilde f_k\}\prod_{i=1}^5d\varsigma_i\notag\\&\hspace{2cm}=\frac{1}{L_\varepsilon}\frac{\partial\Psi_s^{k\tau}}{\partial R_k}\frac{\pi}{\sqrt{\varepsilon | \bm\Theta_k|}}\left[\frac{1}{2}\left(\bm\Theta_k\right)^{-1}+R_k\right]\exp\left[\frac{\varepsilon }{4}\mathbf J_k^T\left(\bm\Theta_k\right)^{-1}\mathbf J_k-\varepsilon \tilde f_k\right],\notag\\
&\frac{1}{L_\varepsilon}\frac{\partial\Psi_s^{k\tau}}{\partial W_k}\int_{\mathbb R^{5K}}\varsigma_5\exp\{-\varepsilon\tilde f_k\}\prod_{i=1}^5d\varsigma_i\notag\\&\hspace{2cm}=\frac{1}{L_\varepsilon}\frac{\partial\Psi_s^{k\tau}}{\partial W_k}\frac{\pi}{\sqrt{\varepsilon | \bm\Theta_k|}}\left[\frac{1}{2}\left(\bm\Theta_k\right)^{-1}+W_k\right]\exp\left[\frac{\varepsilon }{4}\mathbf J_k^T\left(\bm\Theta_k\right)^{-1}\mathbf J_k-\varepsilon \tilde f_k\right].
\end{align}
The system of equations expressed in (\ref{19}) through (\ref{20}) implies that the Wick-rotated Schr\"odinger type equation or the Fokker-Plank type equation is,
\begin{multline*}
\Psi_s^{k\tau}(z_k,S_k,I_k,R_k,W_k)+\varepsilon\frac{\partial \Psi_s^{k\tau}(z_k,S_k,I_k,R_k,W_k)}{\partial s}+o(\varepsilon)\\=\frac{1}{L_\varepsilon}\frac{\pi}{\sqrt{\varepsilon | \bm\Theta_k|}}\exp\left[\frac{\varepsilon }{4}\mathbf J_k^T\left(\bm\Theta_k\right)^{-1}\mathbf J_k-\varepsilon \tilde f_k\right]\biggr\{\Psi_s^{k\tau}(z_k,S_k,I_k,R_k,W_k)\\
+\left[\frac{1}{2}\left(\bm\Theta_k\right)^{-1}+z_k\right]\frac{\partial\Psi_s^{k\tau}(z_k,S_k,I_k,R_k,W_k)}{\partial z_k}\\
+\left[\frac{1}{2}\left(\bm\Theta_k\right)^{-1}+S_k\right]\frac{\partial\Psi_s^{k\tau}(z_k,S_k,I_k,R_k,W_k)}{\partial S_k}\\
+\left[\frac{1}{2}\left(\bm\Theta_k\right)^{-1}+I_k\right]\frac{\partial\Psi_s^{k\tau}(z_k,S_k,I_k,R_k,W_k)}{\partial I_k}\\
+\left[\frac{1}{2}\left(\bm\Theta_k\right)^{-1}+R_k\right]\frac{\partial\Psi_s^{k\tau}(z_k,S_k,I_k,R_k,W_k)}{\partial R_k}\\
+\left[\frac{1}{2}\left(\bm\Theta_k\right)^{-1}+W_k\right]\frac{\partial\Psi_s^{k\tau}(z_k,S_k,I_k,R_k,W_k)}{\partial W_k}\biggr\}+o(\varepsilon^{1/2}),
\end{multline*}
as $\varepsilon\downarrow 0$. Assuming $L_\varepsilon=\pi/\sqrt{\varepsilon | \bm\Theta_k|}>0$ yields,
\begin{multline*}
\Psi_s^{k\tau}(z_k,S_k,I_k,R_k,W_k)+\varepsilon\frac{\partial \Psi_s^{k\tau}(z_k,S_k,I_k,R_k,W_k)}{\partial s}+o(\varepsilon)\\=\left[1+\varepsilon\left(\frac{1 }{4}\mathbf J_k^T\left(\bm\Theta_k\right)^{-1}\mathbf J_k-\varepsilon \tilde f_k\right)\right]\biggr\{\Psi_s^{k\tau}(z_k,S_k,I_k,R_k,W_k)\\
+\left[\frac{1}{2}\left(\bm\Theta_k\right)^{-1}+z_k\right]\frac{\partial\Psi_s^{k\tau}(z_k,S_k,I_k,R_k,W_k)}{\partial z_k}\\
+\left[\frac{1}{2}\left(\bm\Theta_k\right)^{-1}+S_k\right]\frac{\partial\Psi_s^{k\tau}(z_k,S_k,I_k,R_k,W_k)}{\partial S_k}\\
+\left[\frac{1}{2}\left(\bm\Theta_k\right)^{-1}+I_k\right]\frac{\partial\Psi_s^{k\tau}(z_k,S_k,I_k,R_k,W_k)}{\partial I_k}\\
+\left[\frac{1}{2}\left(\bm\Theta_k\right)^{-1}+R_k\right]\frac{\partial\Psi_s^{k\tau}(z_k,S_k,I_k,R_k,W_k)}{\partial R_k}\\
+\left[\frac{1}{2}\left(\bm\Theta_k\right)^{-1}+W_k\right]\frac{\partial\Psi_s^{k\tau}(z_k,S_k,I_k,R_k,W_k)}{\partial W_k}\biggr\}+o(\varepsilon^{1/2}),
\end{multline*}
as $\varepsilon\downarrow 0$. As $z_k\leq\varepsilon/c_1\varsigma_1^2$, assume $|\bm\Theta_k^{-1}|\leq 2c_1\varepsilon(1-\varsigma_1^{-1})$ such that $|(2\bm\Theta_k)^{-1}+z_k|\leq c_1\varepsilon$. In the similar fashion we assume $|(2\bm\Theta_k)^{-1}+S_k|\leq c_2\varepsilon$, $|(2\bm\Theta_k)^{-1}+I_k|\leq c_3\varepsilon$, $|(2\bm\Theta_k)^{-1}+R_k|\leq c_4\varepsilon$ and $|(2\bm\Theta_k)^{-1}+W_k|\leq c_5\varepsilon$. Therefore, $|\bm\Theta_k^{-1}|\leq2\varepsilon\min \left\{c_1(1-\varsigma_1^{-1}),c_2(1-\varsigma_2^{-1}),c_3(1-\varsigma_3^{-1}),c_4(1-\varsigma_4^{-1}),c_5(1-\varsigma_5^{-1})\right\}$ such that $|(2\bm\Theta_k)^{-1}+z_k|\downarrow 0$, $|(2\bm\Theta_k)^{-1}+S_k|\downarrow 0$, $|(2\bm\Theta_k)^{-1}+I_k|\downarrow 0$, $|(2\bm\Theta_k)^{-1}+R_k|\downarrow 0$ and $|(2\bm\Theta_k)^{-1}+W_k|\downarrow 0$. Hence,
\begin{multline*}
\Psi_s^{k\tau}(z_k,S_k,I_k,R_k,W_k)+\varepsilon\frac{\partial \Psi_s^{k\tau}(z_k,S_k,I_k,R_k,W_k)}{\partial s}+o(\varepsilon)\\=(1-\varepsilon)\Psi_s^{k\tau}(z_k,S_k,I_k,R_k,W_k)+o(\varepsilon^{1/2}).
\end{multline*}
Therefore the Fokker-Plank type equation of this pandemic system is,
\begin{equation*}
\frac{\partial \Psi_s^{k\tau}(z_k,S_k,I_k,R_k,W_k)}{\partial s}=-\tilde f_k\times \Psi_s^{k\tau}(z_k,S_k,I_k,R_k,W_k).
\end{equation*}
Finally, the solution of
\begin{equation}\label{21}
-\frac{\partial }{\partial e_k}\tilde f_k[s,e_k(s),\varsigma_1(\tau),\varsigma_2(\tau),\varsigma_3(\tau),\varsigma_4(\tau),\varsigma_5(\tau)]\Psi_s^{k\tau}(z_k,S_k,I_k,R_k,W_k)=0,
\end{equation}
is an optimal ``lock down" intensity of risk-group $k$. Moreover, as $\varsigma_1=z_k(s)-z_k(\tau)$, $\varsigma_2=S_k(s)-S_k(\tau)$, $\varsigma_3=I_k(s)-I_k(\tau)$, $\varsigma_4=R_k(s)-R_k(\tau)$ and $\varsigma_5=W_k(s)-W_k(\tau)$ for all $\varepsilon\downarrow 0$, in Equation (\ref{21}), $\varsigma_i$ for all $i=\{1,...,5\}$ can be replaced by our original state variables. As the transition function $\Psi_s^{k\tau}(z_k,S_k,I_k,R_k,W_k)$ is a solution of the Equation (\ref{21}), the result follows.
 \end{proof}

Theorem \ref{t0} gives the solution of an optimal ``lock-down" intensity for a generalized stochastic pandemic system. Consider a function $$g_k(s,z_k,S_k,I_k,R_k,W_k)\in C^2([0,t]\times\mathbb R^{5K})$$ such that
\begin{multline*}
g_k(s,z_k,S_k,I_k,R_k,W_k)=[sz_k-1-\ln(z_k)]+[sS_k-1-\ln(S_k)]+[sI_k-1-\ln(I_k)]\\
+[sR_k-1-\ln(R_k)]+[sW_k-1-\ln(W_k)],
\end{multline*}
with $\partial g_k/\partial s=z_k+S_k+I_k+R_k+W_k$, $\partial g_k/\partial X_i=s-1/X_i$, $\partial^2 g_k/\partial X_i^2=-1/X_i^2$ and $\partial^2 g_k/\partial X_i\partial X_j=0$, for all $i\neq j$ where $X_i$ is $i^{th}$ state variable for all $i=1,...,5$ and $\ln$ stands for natural logarithm. In other words, $X_1=z_k, X_2=S_k,X_3=I_k, X_4=R_k$ and $X_5=W_k$. Therefore,
\begin{multline*}
\tilde f_k(s,e_k,z_k,S_k,I_k,R_k,W_k)=\exp\{-\rho s\}\sum_{k=1}^K\theta_kz_k(s)\left[N_k-e_k{\tilde{\mathcal A}_k}\right]+\breve\chi_k\check h\varpi_kI_k\\+[sz_k-1-\ln(z_k)]+[sS_k-1-\ln(S_k)]+[sI_k-1-\ln(I_k)]
+[sR_k-1-\ln(R_k)]\\+[sW_k-1-\ln(W_k)]+(z_k+S_k+I_k+R_k+W_k)+\left(s-\frac{1}{z_k}\right)\\\times [\kappa_0(1-e_k)-\kappa_1z_kp(\eta_{k_i},s)]+\left(s-\frac{1}{S_k}\right)\biggr\{\eta N_k-\be^k(e_k,z_k)\frac{S_kI_k}{1+rI_k+\eta N_k}-\tau S_k\\
+\zeta R_k\biggr\}+\left(s-\frac{1}{I_k}\right)\left\{\be^k(e_k,z_k)\frac{S_kI_k}{1+rI_k+\eta N_k}-(\mu+\tau)I_k\right\}+\left(s-\frac{1}{R_k}\right)\\
\times \left[\mu I_k-(\tau +\zeta)e_k R_k\right]+\left(s-\frac{1}{W_k}\right)\left\{\omega_k-\varkappa e_k \mathcal Q(|\omega_k|)\left[\omega_k-\omega_l\right]\right\}\\
-\mbox{$\frac{1}{2}$}\left\{\sigma_0^k(z_k-z_k^*)\frac{1}{z_k^2}+\sigma_2^k(S_k-S_k^*)\frac{1}{S_k^2}+\sigma_3^k(I_k-I_k^*)\frac{1}{I_k^2}+\sigma_4^k(R_k-R_k^*)\frac{1}{R_k^2}\right.\\\left. +\sigma_8^k(\omega_k-\omega_l)\frac{1}{W_k^2}\right\},
\end{multline*}
where
\[
\tilde{\mathcal A}_k=S_k+\left[1-\check h\hat\tau_k\hat F_p-(1-\check h)\tau_k F(p_0)\right]I_k+(1-\breve\tau_k)\tilde\tau_kR_k.
\]
In order to satisfy Equation (\ref{22}) Either $\frac{\partial \tilde f_k}{\partial e_k}=0$ or $\Psi_s^{k\tau}=0$. As $\Psi_s^{k\tau}$ is a wave function, it cannot be zero. Therefore, $\frac{\partial \tilde f_k}{\partial e_k}=0$. After setting the diffusion coefficient of Equation (\ref{1.1}) to zero the optimal lock-down intensity is,
\[
e^*=\left(\frac{{\tilde {\mathcal B}}}{{\tilde{\mathcal C}}}\right)^{\frac{1}{\theta-1}},
\]
where
\begin{multline*}
{\tilde {\mathcal B}}=\exp\{-\rho s\}\sum_{k=1}^K\theta_kz_k{\tilde{\mathcal A}}+\left(s-\frac{1}{z_k}\right)\kappa_0+ \left(s-\frac{1}{R_k}\right)(\tau+\zeta)R_k\\+\left(s-\frac{1}{W_k}\right)\left\{\omega_k-\varkappa e_k \mathcal Q(|\omega_k|)\left[\omega_k-\omega_l\right]\right\}+\mbox{$\frac{1}{2}$}\sigma_8^k(\omega_k-\omega_l)\frac{1}{W_k^2},
\end{multline*}
and,
\[
{\tilde{\mathcal C}}=\theta \beta_2^kM\left(\frac{1}{S_k}-\frac{1}{I_k}\right)\left(\frac{S_kI_k}{1+rI_k+\eta N_k}\right)\left[1-\frac{\kappa_0(z_k)^\gamma}{\kappa_1 p(\eta_{k_i},s)}\right]>0.
\]
The expression $e^*$ represents an optimal lock-down intensity. If all of the state variables attain their optimal value then $e^*$ is a global lock-down intensity.

\section{Discussion}

This paper discuss about a stochastic optimization problem where a policy maker's objective is to minimize a dynamic social cost $\mathbf H_\theta$ subject to a lock-down fatigue dynamics, COVID-19 infection $\beta^k$, a multi-risk SIR model and opinion dynamics of risk-group $k$ where \emph{lock-down} intensity is used as my control variable. Under certain conditions I was able to find out a closed form solution of lock-down intensity $e^*$. First I have subdivided the entire population into $K$ number of age-groups such that every person in a group has homogeneous opinion towards vaccination against COVID-19. As each of these group are vulnerable to the pandemic, I renamed the \emph{age-group} as \emph{risk-group} which is consistent with the literature \citep{acemoglu2020}.  As heterogenous opinion of individuals in a risk-group $k$ concerns with multi-layer network, it would be a future research in this context.

A Feynman-type path integral approach has been used to determine a Fokker-Plank type of equation which reflects the entire pandemic scenario. \emph{Feynman path integral} is a quantization method which uses the quantum \emph{Lagrangian} function, while Schr\"odinger's quantization uses the \emph{Hamiltonian} function \citep{fujiwara2017}. As this path integral approach provides a different view point from Schr\"odinger's quantization,it is very useful tool not only in quantum physics but also in engineering, biophysics, economics and finance \citep{kappen2005,anderson2011,yang2014path,fujiwara2017}. These two methods are believed to be equivalent but, this equivalence has not fully proved mathematically as the mathematical difficulties lie in the fact that the \emph{Feynman path integral} is not an integral by means of a countably additive measure \citep{johnson2000,fujiwara2017}. As the complexity and memory requirements of grid-based partial differential equation (PDE) solvers increase exponentially as the dimension of the system increases, this method becomes impractical in the case with high dimensions \citep{yang2014path}. As an alternative one can use a Monte Carlo scheme and this is the main idea of \emph{path integral control} \citep{kappen2005,theodorou2010,theodorou2011,morzfeld2015}. This \emph{path integral control} solves a class a stochastic control problems with a Monte Carlo method for a HJB equation and this approach avoids the need of a global grid of the domain of HJB equation \citep{yang2014path}. In future research I want to use this approach under $\sqrt{8/3}$ Liouville-like quantum gravity surface \citep{pramanik2021}.

\section*{Appendix}
\subsection*{Proof of Lemma \ref{l0}}
For each optimal solution $z_k^*\in\mathbb F^2$ of Equation (\ref{0}), define a squared integrable progressively measurable process $X(z_k^*)$ by 
\begin{equation}\label{2}
X(z_k^*)_s=z_k(0)+\int_0^t\hat\mu(s,e_k,p,z_k)ds+\int_0^t\sigma_0^k(z_k)dB_0^k(s).
\end{equation} 
I will show that $X(z_k^*)\in\mathbb F^2$. Furthermore, as $z_k^*$ is a solution of Equation (\ref{0}) iff $X(z_k^*)=z_k^*$, I will show that $X$ is the strict contraction of the Hilbert space $\mathbb F^2$. Using the fact that 
\[
|\hat\mu(s,e_k,p,z_k)|^2\leq c_0\left[1+|z_k|^2+|\hat\mu(s,e_k,p,z_k(0))|^2|\right]
\]
yields
\begin{equation}\label{3}
||X(z_k)||^2\leq 4\left[t \E|z_0(k)|^2+\E\int_0^t\bigg|\int_0^s\hat\mu(s',e_k,p,z_k)ds'\bigg|^2ds+t\E\sup_{0\leq s\leq t}\bigg|\int_0^s\sigma_0^k(z_k(s'))dB_0^k(s')\bigg|^2ds\right].
\end{equation}
Assumption \ref{as1} implies $t\E|z_k(0)|^2<\infty$. It will be shown that the second and third terms of the right hand side of the inequality (\ref{3}) are also finite. Assumption \ref{as0} implies,
\begin{multline*}
\E\int_0^t\bigg|\int_0^s\hat\mu(s',e_k,p,z_k)ds'\bigg|^2ds\leq \E\int_0^ts\left(\int_0^s|\hat\mu(s',e_k,p,z_k)|^2ds'\right)ds\\\leq c_0\E\int_0^t s\left(\int_0^s(1+|\hat\mu(s',e_k,p,z_k(0))|^2+|z_k(s)|^2)ds'\right)ds\\\leq c_0t^2\left(1+||\hat\mu(s',e_k,p,z_k(0))||^2+\E\sup_{0\leq s\leq t}|z_k(s)|^2\right)<\infty.
\end{multline*}
Doob's maximal inequality and Lipschitz assumption (i.e. Assumption \ref{as0}) implies,
\begin{multline*}
t\E\sup_{0\leq s\leq t}\bigg|\int_0^s\sigma_0^k(z_k(s'))dB_0^k(s')\bigg|^2ds\leq 4t\E\int_0^t|\sigma_0^k(z_k(s))|^2ds\\\leq 4c_0\E\int_0^t(1+|\sigma_0^k(z_k(0))|^2+|z_k(s)|^2)ds\\\leq 4c_0t^2\left(1+||\sigma_0^k(z_k(0))||^2+\E\sup_{0\leq s\leq t}|z_k(s)|^2\right)<\infty.
\end{multline*}
As $X$ maps $\mathbb F^2$ into itself, I show that it is strict contraction. To do so I change Hilbert norm $\mathbb F^2$ to an equivalent norm. Following \cite{carmona2016} for $a>0$ define a norm on $\mathbb F^2$ by $$||\xi||_a^2=\E\int_0^t\exp(-as)|\xi_s|^sds.$$ If $z_k(s)$ and $y_k(s)$ are generic elements of $\mathbb{F}^2$ where $z_k(0)=y_k(0)$, then
\begin{multline*}
\E|X(z_k(s))-X(y_k(s))|^2\leq 2\E\bigg|\int_0^\tau[\hat\mu(s',e_k,p,z_k(s'))-\hat\mu(s',e_k,p,y_k(s'))]ds\bigg|^2\\+2\E\bigg|\int_0^\tau[\sigma_0^k(z_k(s'))-\sigma_0^k(y_k(s'))]dB_0^k(s')\bigg|^2\\\leq 2\tau\E\int_0^\tau|\hat\mu(s',e_k,p,z_k(s'))-\hat\mu(s',e_k,p,y_k(s'))|^2ds'+2\E\int_0^\tau|\sigma_0^k(z_k(s'))-\sigma_0^k(y_k(s'))|^2ds'\\\leq c_0(1+\tau)\int_0^\tau\E|z_k(s')-y_k(s')|^2ds',
\end{multline*}
by Lipschitz's properties of drift and diffusion coefficients. Hence.
\begin{multline*}
||X(z_k)-X(y_k)||_a^2=\int_0^t\exp(-as)\E|X(z_k(s)-X(y_k(s)))|^2ds\leq c_0t\int_0^t\exp(-as)\int_0^t\E|z_k(s')-y_k(s')|^2ds'ds\\\leq c_0t\int_0^t\exp(-as)ds\int_0^t\E|z_k(s')-y_k(s')|^2ds'\leq\frac{c_0t}{a}||z_k-y_k||_a^2.
\end{multline*}
Furthermore, if $c_0t$ is very large, $X$ becomes a strict contraction. Finally, for $s\in[0,t]$
\begin{multline*}
\E \sup_{0\leq s\leq t}|z_k(s)|^2=\E \sup_{0\leq s\leq t}\bigg|z_k(0)+\int_0^{s'}\hat\mu(r,e_k,p,z_k(r))dr+\int_0^{s'}\sigma_0^k(z_k(r))dB_0^k(r)\bigg|^2\\\leq 4\left[\E|z_k(0)|^2+s\E\int_0^s|\hat\mu(s',e_k,p,z_k(s'))|^2ds'+4\E\int_0^s|\sigma_0^k(s')|ds'\right]\\\leq c_0\left[1+\E|z_k(0)|^2+\int_0^s\E\sup_{0\leq r\leq s'}|z_k(r)|^2dr\right],
\end{multline*}
where the constant $c_0$ depends on $t$, $||\hat\mu||^2$ and $||\sigma_0^k||^2$. Gronwall's inequality implies,
\begin{equation*}
\E \sup_{0\leq s\leq t}|z_k(s)|^2\leq c_0(1+\E|z_k(0)|^2)\exp{(c_0t)}.
\end{equation*}
Q.E.D.

\subsection*{Proof of Proposition \ref{p0}}
As stochastic differential Equation (\ref{0}) and the SIR represented by the system (\ref{4}) follow Assumption \ref{as0}, there is a unique local solution on continuous time interval $[0,\hat s)$, where $\hat s$ is defined as the explosion point \citep{rao2014}. Therefore, It\^o formula makes sure that there is a positive unique local solution for the system represented by Equations (\ref{0}) and (\ref{4}). In order to show global uniqueness one needs to show this local unique solution is indeed a global solution; in other words, $\hat s=\infty$ almost surely.

Suppose, $m_0>0$ is sufficiently large for the initial values of the state variables $z_k(0)$, $S_k(0)$, $I_k(0)$ and $R_k(0)$ in the interval $[1/m_0,m_0]$. For all $m\geq m_0$ a sequence of stopping time is defined as
\begin{multline*}
\hat s_m=\inf\left\{s\in[0,\hat s]:z_k(s)\notin\left(\frac{1}{m},m\right)\text{or}\ z_k(s)\notin\left(\frac{1}{m},m\right)\text{or}\ S_k(s)\notin\left(\frac{1}{m},m\right)\right.\\\left.\text{or}\ I_k(s)\notin\left(\frac{1}{m},m\right) \text{or}\ R_k(s)\notin\left(\frac{1}{m},m\right)\right\},
\end{multline*}
where it is assumed that the infimum of the empty set is infinity. As the explosion time is non-decreasing in $m$ therefore, $\hat s_\infty=\lim_{m\downarrow \infty}\hat s_m$ and $\hat s_{\infty} \leq \hat s_m$ a.s. I will show $\hat s_\infty=\infty$ a.s. Suppose that the condition $\hat s_\infty=\infty$ a.s. does not hold. Then  $\exists$ a $t>0$ and $\varepsilon>0$ such that $Pr[\hat s_\infty\leq t]>\varepsilon$. Hence, there is an integer $m_1\geq m_0$ such that, $Pr[\hat s_m\leq t]\geq \varepsilon,\ \forall m\geq m_1$.

Like before, define a non-negative $C^3$-function $\mathfrak W:\mathbb R^{4K}\ra \mathbb R$ by
\begin{equation*}
\mathfrak W(z_k,S_k,I_k,R_k)=[z_k-1-\ln (z_k)]+[S_k-1-\ln (S_k)]+[I_k-1-\ln (I_k)]+[R_k-1-\ln (R_k)].
\end{equation*}
It\^o's formula implies
\begin{multline*}
d \mathfrak W(z_k,S_k,I_k,R_k)=\left\{\left(1-\frac{1}{z_k}\right)\left[\kappa_0(1-e_k)-\kappa_1z_kp(\eta_{k_i})\right]+\left(1-\frac{1}{S_k}\right)\right.\\\left.\times\left[\eta N_k-\be^k(e_k,z_k)\frac{S_kI_k}{(1+rI_k)+\eta N_k}-\tau S_k+\zeta R_k\right]+\left(1-\frac{1}{I_k}\right)\right.\\\left.\times\left[\be^k(e_k,z_k)\frac{S_kI_k}{\left[1+rI_k\right]+\eta N_k}-(\mu+\tau)I_k\right]+\left(1-\frac{1}{R_k}\right)\left[\mu I_k-(\tau +\zeta)e_k R_k\right]\right.\\\left.+ \frac{(\sigma_0^k)^2}{2}\left(1-\frac{z_k^*}{z_k}\right)^2+\frac{(\sigma_2^k)^2}{2}\left(1-\frac{S_k^*}{S_k}\right)^2+\frac{(\sigma_3^k)^2}{2}\left(1-\frac{I_k^*}{I_k}\right)^2\right.\\\left.+\frac{(\sigma_4^k)^2}{2}\left(1-\frac{R_k^*}{R_k}\right)^2\right\}ds+\left\{\sigma_0^k\left(1-\frac{1}{z_k}\right)(z_k-z_k^*)+\sigma_2^k\left(1-\frac{1}{S_k}\right)(S_k-S_k^*)\right.\\\left.+\sigma_3^k\left(1-\frac{1}{I_k}\right)(I_k-I_k^*)+\sigma_4^k\left(1-\frac{1}{R_k}\right)(R_k-R_k^*)\right\}dB^k,
\end{multline*}
where I assume $B^k=B_0^k=B_2^k=B_3^k=B_4^k$ or the system has same Brownian motion.
Therefore,
\begin{multline}\label{11}
d \mathfrak W(z_k,S_k,I_k,R_k)=\left\{\zeta R_k+\eta N_k+\mu(1+I_k)+\tau(1+R_k)+\kappa_0\left(1+\frac{e_k}{z_k}\right)\right.\\\left.+\kappa_1p(\eta_{k_i})+I_k\left(\frac{\be^k(e_k,z_k)}{\left[1+rI_k\right]+\eta N_k}+\frac{\mu+\tau}{S_k}\right)+\be^k(e_k,z_k)\frac{S_kI_k}{\left[1+rI_k\right]+\eta N_k}\right.\\\left. +\frac{(\sigma_0^k)^2}{2}\left(1-\frac{z_k^*}{z_k}\right)^2+\frac{(\sigma_2^k)^2}{2}\left(1-\frac{S_k^*}{S_k}\right)^2+\frac{(\sigma_3^k)^2}{2}\left(1-\frac{I_k^*}{I_k}\right)^2\right.\\\left.+\frac{(\sigma_4^k)^2}{2}\left(1-\frac{R_k^*}{R_k}\right)^2-\left[(\tau+\zeta)e_kR_k+2(\mu+\tau)I_k+\kappa_1z_kp(\eta_{k_i})+\frac{\mu I_k}{R_k}\right.\right.\\\left.\left.+\frac{\eta N_k}{S_k}+\kappa_0\left(e_k+\frac{1}{z_k}\right)+\frac{S_k\be(e_k,z_k)}{\left[1+rI_k\right]+\eta N_k}(1+I_k)\right]\right\}ds\\+\left\{\sigma_0^k\left(1-\frac{1}{z_k}\right)(z_k-z_k^*)+\sigma_2^k\left(1-\frac{1}{S_k}\right)(S_k-S_k^*)\right.\\\left.+\sigma_3^k\left(1-\frac{1}{I_k}\right)(I_k-I_k^*)+\sigma_4^k\left(1-\frac{1}{R_k}\right)(R_k-R_k^*)\right\}dB^k\\\leq \left\{\zeta R_k+\eta N_k+\mu(1+I_k)+\tau(1+R_k)+\kappa_1p(\eta_{k_i})+\kappa_0\left(1+\frac{e_k}{z_k}\right)\right.\\\left.+I_k\left(\frac{\be^k(e_k,z_k)}{\left[1+rI_k\right]+\eta N_k}+\frac{\mu+\tau}{S_k}\right)+\be^k(e_k,z_k)\frac{S_kI_k}{\left[1+rI_k\right]+\eta N_k}\right.\\\left. +\frac{(\sigma_0^k)^2}{2}\left(1-\frac{z_k^*}{z_k}\right)^2+\frac{(\sigma_2^k)^2}{2}\left(1-\frac{S_k^*}{S_k}\right)^2+\frac{(\sigma_3^k)^2}{2}\left(1-\frac{I_k^*}{I_k}\right)^2\right.\\\left.+\frac{(\sigma_4^k)^2}{2}\left(1-\frac{R_k^*}{R_k}\right)^2\right\}ds+\left\{\sigma_0^k\left(1-\frac{1}{z_k}\right)(z_k-z_k^*)+\sigma_2^k\left(1-\frac{1}{S_k}\right)\right.\\\left.\times(S_k-S_k^*)+\sigma_3^k\left(1-\frac{1}{I_k}\right)(I_k-I_k^*)+\sigma_4^k\left(1-\frac{1}{R_k}\right)(R_k-R_k^*)\right\}dB^k\\\leq\mathfrak M\  ds+\left\{\sigma_0^k\left(1-\frac{1}{z_k}\right)(z_k-z_k^*)+\sigma_2^k\left(1-\frac{1}{S_k}\right)(S_k-S_k^*)\right.\\\left.+\sigma_3^k\left(1-\frac{1}{I_k}\right)(I_k-I_k^*)+\sigma_4^k\left(1-\frac{1}{R_k}\right)(R_k-R_k^*)\right\}dB^k,
\end{multline}
where $\mathfrak M$ is a positive constant. Integration of both sides of the Inequality (\ref{11}) from $0$ to $\hat s_m\wedge t$ yield
\begin{align*}
&\int_0^{\hat s_m\wedge t}d \mathfrak W[z_k(s),S_k(s),I_k(s),R_k(s)]\\ &\hspace{1cm}\leq \int_0^{\hat s_m\wedge t}\mathfrak M ds+\left\{\sigma_0^k\left(1-\frac{1}{z_k}\right)(z_k-z_k^*)+\sigma_2^k\left(1-\frac{1}{S_k}\right)(S_k-S_k^*)\right.\\&\left.\hspace{2cm}+\sigma_3^k\left(1-\frac{1}{I_k}\right)(I_k-I_k^*)+\sigma_4^k\left(1-\frac{1}{R_k}\right)(R_k-R_k^*)\right\}dB^k,
\end{align*}
where $\hat s_m\wedge t=\min\{\hat s_m, t\}$. After taking expectations on both sides lead to 
\begin{equation*}
\E\mathfrak W[z_k(\hat s_m\wedge t),S_k(\hat s_m\wedge t),I_k(\hat s_m\wedge t),R_k(\hat s_m\wedge t)]\leq\mathfrak M t+\mathfrak W[z_k(0),S_k(0),I_k(0),R_k(0)].
\end{equation*}
Define $\aleph_m=\{\hat s_m\leq t\},\ \forall m\geq m_1$. Previous discussion implies, for any $\varepsilon>0$ there exists an integer $m_1\geq m_0$ such that, $Pr[\hat s_m\leq t]\geq \varepsilon$ therefore, $Pr(\aleph_m)\geq\varepsilon$. For each $\wp\in\aleph_m$,  $\exists$ an $i$ such that $\hbar_i(\hat s_m,\wp)=m\ \text{or}\ 1/m$ for $i=1,...,4$. Therefore, $\mathfrak W\left[z_k(\hat s_m,\wp),S_k(\hat s_m,\wp),I_k(\hat s_m,\wp),R_k(\hat s_m,\wp)\right]$ has the lower bound $\min\{m-1-\ln m, 1/m-1-\ln(1/m)\}$. This yields,
\begin{eqnarray*}
& &\mathfrak M t+\mathfrak W[z_k(0),S_k(0),I_k(0),R_k(0)]\\ & & \geq \E\left\{\mathbbm 1_{\aleph_m(\wp)}\mathfrak W[z_k(\hat s_m),S_k(\hat s_m),I_k(\hat s_m),R_k(\hat s_m)]\right\}\\ & & \geq \varepsilon \min\left\{m-1-\ln (m),\frac{1}{m}-1-\ln\left(\frac{1}{m}\right)\right\},
\end{eqnarray*}
where $\mathbbm 1_{\aleph_m(\wp)}$ is a simple function on $\aleph_m$. Letting $m\downarrow\infty$ leads to $\infty=\mathfrak M t+\mathfrak W[z_k(0),S_k(0),I_k(0),R_k(0)]<\infty$, which is a contradiction. Q.E.D.

\subsection*{Proof of Lemma \ref{l1}}
As stochastic opinion dynamics is on $F$, this surface is oscillatory in nature. Total social interaction variation between two probabilistic interactions\\ $W_{k,\chi_{-k},\omega_{-k}}(s,h)$ and $W_{l,\chi_{-l},\omega_{-l}}(s,h)$ can be defined in terms of a Hahn-Jordon orthogonal decomposition 
\[
W=W_{k,\chi_{-k},\omega_{-k}}-W_{l,\chi_{-l},\omega_{-l}}=W_{k,\chi_{-k},\omega_{-k}}^+-W_{l,\chi_{-l},\omega_{-l}}^-,
\]
such that 
\[
\left|\left|\left(W_{k,\chi_{-k}\omega_{-k}}-W_{l,\chi_{-l},\omega_{-l}}\right)\right|\right|=W_{k,\chi_{-k},\omega_{-k}}^+(F)=W_{l,\chi_{-l},\omega_{-l}}^-(F).
\]
Therefore, for $h\in F$,
\begin{multline*}
\left|W_{k,\chi_{-k},\omega_{-k}}(s,h)-W_{l,\chi_{-l},\omega_{-l}}(s,h)\right|\\ =\left|\int_{\mathbb R}h(s,\omega_k)W_{k,\chi_{-k},\omega_{-k}}^+(F)(d\omega_k)-\int_{\mathbb R}h(s,\omega_k)W_{l,\chi_{-l},\omega_{-l}}^-(F)(d\omega_l)\right|\\ =\left|\left|\left(W_{k,\chi_{-k}\omega_{-k}}(s)-W_{l,\chi_{-l},\omega_{-l}}(s)\right)\right|\right|\left|\int_{\mathbb R}\left[h(s,\omega_k)-h(s,\omega_l)\right]\right.\\\left.\times\frac{W_{k,\chi_{-k},\omega_{-k}}^+(d\omega_k)}{W_{k,\chi_{-k},\omega_{-k}}^+(F)}\times\frac{W_{l,\chi_{-l},\omega_{-l}}^-(d\omega_l)}{W_{l,\chi_{-l},\omega_{-l}}^-(F)}\right|.
\end{multline*}
Therefore,
\[
\left|W_{k,\chi_{-k},\omega_{-k}}(s,h)-W_{l,\chi_{-l},\omega_{-l}}(s,h)\right|\leq \left|\left|\left(W_{k,\chi_{-k}\omega_{-k}}-W_{l,\chi_{-l},\omega_{-l}}\right)\right|\right|.
\]
Supremum over $h\in F$ yields,
\[
\sup\left\{\left|\left(W_{k,\chi_{-k},\omega_{-k}}(s,h)-W_{l,\chi_{-l},\omega_{-l}}(s,h)\right)\right|\right\}\leq \left|\left|\left(W_{k,\chi_{-k}\omega_{-k}}-W_{l,\chi_{-l},\omega_{-l}}\right)\right|\right|.
\]
The reverse inequality can be checked trivially by introducing a simple function $\mathbbm 1_{\mathcal G}$, with $\mathcal G\in \mathcal E$, belong to F. Therefore, we are able to show that
\begin{equation*}
\left|\left|\left(W_{k,\chi_{-k}\omega_{-k}}(s)-W_{l,\chi_{-l},\omega_{-l}}(s)\right)\right|\right|= \sup\left\{\left|\left(W_{k,\chi_{-k},\omega_{-k}}(s,h)-W_{l,\chi_{-l},\omega_{-l}}(s,h)\right)\right|\right\}.
\end{equation*}
Now, by construction , there exists two disjoint subsets $F_+$ and $F_-$ such that, $W^+(F_-)=0=W^-(F_+)$ \citep{moral2004}. For any graph $\mathcal G\in \mathcal E$, $W^+(\mathcal G)=W(\mathcal G \cap F_+)\geq 0$ and, $W^-(\mathcal G)=-W(\mathcal G \cap F_+)\geq 0.$
Hence,
\[
W_{k,\chi_{-k},\omega_{-k}}(\mathcal G \cap F_+)\geq W_{l,\chi_{-l},\omega_{-l}}(\mathcal G \cap F_+),
\]
and,
\[
W_{l,\chi_{-l},\omega_{-l}}(\mathcal G \cap F_-)\geq W_{k,\chi_{-k},\omega_{-k}}(\mathcal G \cap F_-).
\]
Consider $\hat h$ be another probability measure for any $\mathcal G\in \mathcal E$ by,
\[
\hat h(\mathcal G)=W_{k,\chi_{-k},\omega_{-k}}(\mathcal G \cap F_-)+W_{l,\chi_{-l},\omega_{-l}}(\mathcal G \cap F_+).
\]
By construction, 
\begin{equation}\label {9}
\hat h(\mathcal G)\leq W_{k,\chi_{-k},\omega_{-k}}(\mathcal G)\wedge W_{l,\chi_{-l},\omega_{-l}}(\mathcal G),
\end{equation}
and,
\begin{equation}\label{10}
\hat h(F)=W_{k,\chi_{-k},\omega_{-k}}( F_-)+W_{l,\chi_{-l},\omega_{-l}}(F_+).
\end{equation}
As
\begin{multline*}
\left|\left|\left(W_{k,\chi_{-k}\omega_{-k}}(s)-W_{l,\chi_{-l},\omega_{-l}}(s)\right)\right|\right|=W^+(F)=W(F_+)\\=W_{k,\chi_{-k},\omega_{-k}}( F_+)-W_{l,\chi_{-l},\omega_{-l}}(F_-)=1-\left[W_{k,\chi_{-k},\omega_{-k}}( F_+)+W_{l,\chi_{-l},\omega_{-l}}(F_-)\right],
\end{multline*}
by Equation (\ref{10}) one obtains
\begin{equation*}
1-\sup_{\hat h\in \left(W_{k,\chi_{-k},\omega_{-k}}(s),W_{l,\chi_{-l},\omega_{-l}}(s)\right)} \tilde h(F)\leq 1-\hat h(F)=\left|\left|\left(W_{k,\chi_{-k}\omega_{-k}}(s)-W_{l,\chi_{-l},\omega_{-l}}(s)\right)\right|\right|.
\end{equation*}
The reverse inequality is proved as follows. Suppose, $\tilde h$ be a non-negative measure such that for any graph $\mathcal G\in \mathcal E$ we have
\[
\tilde h(\mathcal G)\leq W_{k,\chi_{-k}\omega_{-k}}(\mathcal G) \wedge W_{l,\chi_{-l},\omega_{-l}}(\mathcal G). 
\]
Assuming $\mathcal G=F_+$ and $\mathcal G=F_-$ give us
\[
\tilde h (F_+)\leq W_{k,\chi_{-k},\omega_{-k}}(F_+)\ \text{and,}\ \tilde h (F_-)\leq W_{l,\chi_{-l},\omega_{-l}}(F_-).
\]
Therefore,
\[
\tilde h(F)\leq W_{k,\chi_{-k},\omega_{-k}}(F_+)+W_{l,\chi_{-l},\omega_{-l}}(F_-)=1-\left|\left|\left(W_{k,\chi_{-k}\omega_{-k}}(s)-W_{l,\chi_{-l},\omega_{-l}}(s)\right)\right|\right|,
\]
which implies
\[
1-\tilde h(F)\geq \left|\left|\left(W_{k,\chi_{-k}\omega_{-k}}(s)-W_{l,\chi_{-l},\omega_{-l}}(s)\right)\right|\right|.
\]
Taking the infimum over all the distributions $\tilde h \leq W_{k,\chi_{-k}\omega_{-k}}(s)$ and $W_{l,\chi_{-l},\omega_{-l}}(s)$, we get 
\[
\left|\left|\left(W_{k,\chi_{-k}\omega_{-k}}(s)-W_{l,\chi_{-l},\omega_{-l}}(s)\right)\right|\right|=1-\sup_{\hat h\in \left(W_{k,\chi_{-k},\omega_{-k}}(s),W_{l,\chi_{-l},\omega_{-l}}(s)\right)} \hat h(F).
\] To prove the final part of the lemma note that,
\[
W_{l,\chi_{-l},\omega_{-l}}(F_+)=W_{k,\chi_{-k},\omega_{-k}}(F_+)\wedge W_{l,\chi_{-l},\omega_{-l}}(F_+),
\]
and 
\[
W_{k,\chi_{-k},\omega_{-k}}(F_-)=W_{k,\chi_{-k},\omega_{-k}}(F_-)\wedge W_{l,\chi_{-l},\omega_{-l}}(F_-).
\]
Hence,
\begin{multline*}
\hat h(F)=W_{k,\chi_{-k},\omega_{-k}}(F_-)+W_{l,\chi_{-l},\omega_{-l}}(F_+)\\=\left[W_{k,\chi_{-k},\omega_{-k}}(F_-)\wedge W_{l,\chi_{-l},\omega_{-l}}(F_-)\right]+\left[W_{k,\chi_{-k},\omega_{-k}}(F_+)\wedge W_{l,\chi_{-l},\omega_{-l}}(F_+)\right].
\end{multline*}
As $F_+$ and $F_-$ are mutually exclusive, therefore,
\[
\hat h(F)\geq\inf\sum_{i=1}^I\left[W_{k,\chi_{-k},\omega_{-k}}(\mathcal G_i)\wedge W_{l,\chi_{-l},\omega_{-l}}(\mathcal G_i) \right],
\]
where the infimum is taken over all resolutions of $F$ into pairs of nonintersecting subgraphs $\mathcal G_i$, $1\leq i\leq I$, $I\geq 1$. Reverse inequality can be shown by using the definition of $\hat h$. By Equation (\ref{10}) for any finite subgraph $\mathcal G_i\in \mathcal E$, we have 
\[
\hat h(\mathcal G_i)\leq W_{k,\chi_{-k},\omega_{-k}}(\mathcal G_i)\wedge W_{l,\chi_{-l},\omega_{-l}}(\mathcal G_i).
\]
Therefore,
\[
\hat h(F)=\sum_{i=1}^I\hat h(\mathcal G_i)\leq \sum_{i=1}^I\left[W_{k,\chi_{-k},\omega_{-k}}(\mathcal G_i)\wedge W_{l,\chi_{-l},\omega_{-l}}(\mathcal G_i)\right]
\]
By taking the infimum over all subgraphs yields \[
\left|\left|\left(W_{k,\chi_{-k}\omega_{-k}}(s)-W_{l,\chi_{-l},\omega_{-l}}(s)\right)\right|\right|=1-\inf\sum_{i=1}^I\left(W_{k,\chi_{-k},\omega_{-k}}(s,\mathcal G_i)\wedge W_{l,\chi_{-l},\omega_{-l}}(s,\mathcal G_i)\right),
\]
since 
\[
\hat h(F)=1-\left|\left|\left(W_{k,\chi_{-k}\omega_{-k}}(s)-W_{l,\chi_{-l},\omega_{-l}}(s)\right)\right|\right|.
\]
This completes the proof. Q.E.D.

\subsection*{Proof of Proposition \ref{p1}}

Consider $\mho:\Omega_E\ra\mathbb E$ is an increasing function which represents the influence of risk-group $k$ in the network which is a convex function of the odds of themselves to get the signals from the neighbors about their social interactions and is defined by $\hat\pi=\log[\varrho/(1-\varrho)]$. Assume for $\mathfrak q_3$, the signal profile $\varrho$ of risk-group $k$ is in  $\mathbb D\subseteq(0,1)$. Now suppose, $\mathcal J$ is the total number of interactions of risk-group $k$ with open edges with the set of edges with boxes $\Lambda$ as $\mathbb E_\Lambda$. Then by Theorem $4.2$ of \cite{grimmett1995} and by Picard-Lindelof theorem there exists a unique random opinion in $G$ \citep{board2021}. Q.E.D.

\subsection*{Proof of Proposition \ref{p2}}
I have divided the proof into two cases.

$\mathbf{Case\ I}$: There are total $K$-risk-groups with an individual risk-group $k$ such that $k=1,2,...,K$. I assume that $m\subset \mathbb N$, a set $\beth$ with condition $|\beth|=m+1$, and affinely independent state variables and lock-down intensity $\{Z_k(s)\}_{k\in \beth}\subset \mathbb R^{6K}\times G$ such that $\widetilde \Xi$ coincides with the simplex convex set of $\{ Z_k(s)\}_{k\in \beth}$. For each $Z(s)\subset \Xi$, there is a unique way in which the vector $Z(s)$ can be written as a convex combination of the extreme valued state variables and lock-down intensity , namely, $Z(s)=\sum_{k\in\beth}\alpha_k(s,Z)Z_k(s)$ such that $\sum_{k\in\beth}\alpha_k(s,Z)=1$ and $\alpha_k(s,Z)\geq 0,\ \forall k\in\beth$ and $s\in[0,t]$. For each risk-group $k$, define a set 
\[
\widetilde\Xi_k:=\left\{Z\in\widetilde\Xi:\a_k[\mathcal L_k(s,Z)]\leq\a_k(s,Z)\right\}.
\]
By the continuity of the quantum Lagrangian of $k^{th}$ risk-group $\{\mathcal L_k\}_{k\in\beth}$, $\widetilde\Xi_k$ is closed. Now we claim that, for every $\tilde \beth \subset \beth$, the convex set consists of $\{Z_k\}_{k\in\tilde\beth}$ is proper subset of $\bigcup_{k\in\tilde\beth}\widetilde\Xi_k$. Suppose $\tilde\beth\subset\beth$ and $Z(s)$ is also in the non-empty, convex set consists of the state variables and the lock-down intensity $\{Z_k(s)\}_{k\in\tilde\beth}$. Therefore, there exists $k\in\tilde\beth$ such that $\a_k(s,Z)\geq\a_k\left[\mathcal L_k(s,Z)\right]$ which implies $Z(s)\in\tilde\Xi\subset \bigcup_{l\in\tilde\beth}\tilde\Xi_l$. By \emph{Knaster-Kuratowski-Mazurkiewicz Theorem}, there is $\bar Z_k^*\in \bigcap_{k\in\beth}\tilde\Xi_k$, in other words, the condition $\a_k\left[\mathcal L_k(s,\bar Z_k^*)\right]\leq\a_k(s,\bar Z_k^*)$ for all $k\in\beth$ and for each $s\in[0,t]$ \citep{gonzalez2010}. Hence, $\mathcal L_k(s,\bar Z_k^*)=\bar Z_k^*$ or $\mathcal L_k$ has a fixed-point. 

$\mathbf{Case\ II}$: Again consider $\widetilde\Xi\subset \mathbb R^{6K}\times G$ is a non-empty, convex and compact set. Then for $m\subset \mathbb N$, a set $\beth$ with condition $|\beth|=m+1$, and affinely independent state variables and lock-down intensity $\{Z_k(s)\}_{k\in \beth}\subset \mathbb R^{6K}\times G$ such that $\widetilde\Xi$ is a proper subset of the convex set based on $\{Z_k(s)\}_{k\in\beth}$ for all $s\in[0,t]$. Among all the simplices, suppose $\hat\aleph$ is the set with smallest $m$. Let $\tilde Z(s)$ be a dynamic point in the $m$-dimensional interior of $\hat\aleph$. Define ${\hat{\mathcal L}}_k$, an extension of $\mathcal L_k$ to the whole simplex $\hat\aleph$, as follows. For every $Z(s)\in\hat\aleph$, let 
\[
\bar\zeta(s,Z):\max\left\{\bar\zeta\in[0,1]:(1-\bar\zeta)\tilde Z(s)+\bar\zeta Z(s)\in\widetilde\Xi\right\},\ \forall s\in[0,1],
\]
and,
\[
{\hat{\mathcal L}}_k(s,Z):\mathcal L_k\left\{\left[1-\bar\zeta(s,Z)\right]\tilde Z(s)+\bar\zeta(s,Z) Z(s)\right\}.
\]
Therefore, $\bar\zeta$ is continuous which implies ${\hat{\mathcal L}}_k(s,Z)$ is continuous. Since the codomain of ${\hat{\mathcal L}}_k(s,Z)$ is in $\tilde\Xi$, every fixed-point of ${\hat{\mathcal L}}_k(s,Z)$ is also a fixed-point of $\mathcal L_k$. Now by \emph{$\mathbf{Case\  I}$}, ${\hat{\mathcal L}}_k(s,Z)$ has a fixed-point and therefore, $\mathcal L_k$ also does. Q.E.D.

\section*{Funding declaration}
No funding was used to write this paper.

\section*{Conflict of interest}
The author declares that he has no conflicts of interest.

\section*{Author contribution}
The author declares that  he has solely contributed the whole paper.
\bibliography{bib}
\end{document}